\documentclass[a4paper]{amsart}

\usepackage{hyperref}

\usepackage{tikz}
\usetikzlibrary{patterns}
\usetikzlibrary{cd}
\tikzcdset{arrow style=tikz}

\usepackage{amssymb,amsrefs,amsthm,mathtools}

\usepackage[marginpar]{todo}

\theoremstyle{plain}
\newtheorem{theorem}{Theorem}[section]
\newtheorem{lemma}[theorem]{Lemma}
\newtheorem{proposition}[theorem]{Proposition}
\newtheorem{corollary}[theorem]{Corollary}
\newtheorem{resultx}{Result}

\theoremstyle{definition}
\newtheorem{definition}[theorem]{Definition}

\newtheorem{observation}[theorem]{Observation}

\newtheorem*{question*}{Question}
\newtheorem{question}[theorem]{Question}
\theoremstyle{remark}
\newtheorem{example}[theorem]{Example}
\newtheorem{remark}[theorem]{Remark}
\newcommand{\funct}[1]{\operatorname{\mathtt{#1}}}

\DeclareMathOperator{\Image}{Im} \renewcommand{\Im}{\Image}
\DeclareMathOperator{\End}{End}
\DeclareMathOperator{\Hom}{Hom}

\DeclareMathOperator{\Endo}{End} \renewcommand{\End}{\Endo}

\newcommand{\iso}{\cong}

\newcommand{\Top}{\mathbf{Top}}
\newcommand{\Set}{\mathbf{Set}}
\newcommand{\set}{\mathbf{set}}
\newcommand{\pSet}{\mathbf{Set}^\mathrm{pt}}
\newcommand{\pset}{\mathbf{set}^\mathrm{pt}}
\newcommand{\Vect}{\mathbf{Vect}}
\newcommand{\vect}{\mathbf{vect}}
\newcommand{\VectF}{\mathbf{Vect}_F}
\newcommand{\vectF}{\mathbf{vect}_F}
\newcommand{\Z}{\mathbb{Z}}
\DeclareMathOperator{\rep}{rep}
\DeclareMathSymbol{\mhyphen}{\mathord}{AMSa}{"39}

\title{On the additive image of 0th persistent homology}

\author{Ulrich Bauer}
\address{Department of Mathematics, Technical University of Munich, Garching bei München, Germany}
\email{mail@ulrich-bauer.org}

\author{Magnus B.\ Botnan}
\address{Department of Mathematics, Vrije Universiteit Amsterdam, Amsterdam, The Netherlands}
\email{m.b.botnan@vu.nl}

\author{Steffen Oppermann}
\address{Department of Mathematical Sciences, Norwegian University of Science and Technology, Trondheim, Norway}
\email{steffen.oppermann@ntnu.no}

\author{Johan Steen}
\address{Saddle AS, Stavanger, Norway}
\email{johan.steen.jr@gmail.com}
\begin{document}
\begin{abstract}
For a category $X$ and a finite field $F$, we study the additive image of the functor $\operatorname{H}_0(-;F)_* \colon \rep(X, \Top) \to \rep(X, \VectF)$,
or equivalently, of the free functor $\rep(X, \Set) \to \rep(X, \VectF)$.

We characterize all finite categories \( X \) for which the indecomposables in the additive image coincide with the indecomposable indicator representations and provide examples of quivers of wild representation type where the additive image contains only finitely many indecomposables. Motivated by questions in topological data analysis, we conduct a detailed analysis of the additive image for finite grids. In particular, we show that for grids of infinite representation type, there exist infinitely many indecomposables both within and outside the additive image.

We develop an algorithm for determining if a representation of a finite category is in the additive image. In addition, we investigate conditions for realizability and the effect of modifications of the source category and the underlying field. 

The paper concludes with a discussion of the additive image of $\operatorname{H}_n(-;F)_*$ for an arbitrary field $F$, extending previous work for prime fields.

\end{abstract}
\maketitle

\section{Introduction}
Topological data analysis (TDA) is a branch of data science which applies topology to study the shape of data, i.e., the coarse-scale, global, non-linear geometric features of data. The most notable technique in TDA is persistent homology, and the  pipeline for persistent homology can generally be divided into three steps:
\begin{enumerate}
\item Construct a commutative diagram of topological spaces and inclusions from the data. 
\item Apply homology with field coefficients to obtain a linear representation, i.e., a commutative diagram of vector spaces. 
\item Compute descriptors from the representation in (2) for the use in data analysis.
\end{enumerate}
In the simplest setting, one considers a finite filtration of simplicial complexes \[K_1 \subseteq K_2 \subseteq \cdots \subseteq K_m.\]
Applying homology with coefficients in a field $F$ yields the following sequence of vector spaces and linear maps, \[\operatorname{H}_p(K_1;F) \to \operatorname{H}_p(K_2;F)\to \cdots \to \operatorname{H}_p(K_m;F),\]
which is a representation of a quiver of type $A_m$.
As is well-known, such representations decompose into a direct sum of representations of the form 
\[ \cdots \to 0 \to F \xrightarrow{1} F \xrightarrow{1} \cdots \xrightarrow{1} F \to 0 \to 0\to \cdots.\]
A corresponding result holds for filtrations over arbitrary totally ordered sets if the vector spaces are finite-dimensional. The supports of the intervals constitute the \emph{barcode} used in Step (3) as input for exploratory and statistical data analysis and machine learning methods; see \cite{chazal2021introduction} for an introduction to TDA and persistent homology.

There are, however, many situations where it is desirable to filter the data by more than one parameter. Examples include data with outliers or variations in density, time-varying data, and data that naturally comes equipped with a real-valued function; we refer the reader to \cite{botnan2022introduction} for a recent introduction to multiparameter persistent homology (MPH). The aforementioned settings naturally lead to \emph{bifiltrations} along a product $P$ of totally ordered sets, i.e., a family of topological spaces $\{D_p\}_{p\in P}$ satisfying $D_p\subseteq D_q$ for $p\leq q$. Of particular interest is the product (in the category of posets)
\[
A_{m_1}\otimes A_{m_2} :=
\begin{tikzcd}[row sep=4mm, column sep=5mm]
\scriptstyle (1, m_2) \ar[r] & \scriptstyle (2, m_2) \ar[r] &  \cdots\ar[r] & \scriptstyle (m_1, m_2) \\
\vdots \ar[u] & \vdots \ar[u] && \vdots \ar[u] \\
\scriptstyle (1,2) \ar[r]  \ar[u] & \scriptstyle (2,2) \ar[u]\ar[r] & \cdots \ar[r] & \scriptstyle (m_1, 2) \ar[u] \\
\scriptstyle (1,1) \ar[r]  \ar[u] & \scriptstyle (2,1) \ar[u]\ar[r] & \cdots \ar[r] & \scriptstyle (m_1, 1) \ar[u] 
\end{tikzcd}
\]
For
$(m_1-1) \times (m_2-1) > 4$, 
the poset
$A_{m_1}\otimes A_{m_2}$ 
is of wild representation type, and, in particular, there does not seem to be any meaningful way of using indecomposable representations in Step (3) in the above pipeline. Furthermore, as first observed by Carlsson and Zomorodian \cite{carlsson2007theory}*{Theorem~2}, this algebraic complexity is also present in representations coming from topology: for a prime field $F$, every linear representation of $A_{m_1}\otimes A_{m_2}$ can be obtained by applying $\operatorname{H}_1(-;F)_*$ to a bifiltration of simplicial or cellular complexes. For instance, the following indecomposable representation
\begin{equation}
\label{eq.intro-comm}
R = 
\begin{tikzcd}[column sep=15mm, ampersand replacement=\&]
F \ar{r}{\left(\begin{smallmatrix} 0 \\ 1 \end{smallmatrix}\right)} \& F^2\ar{r}{\left(\begin{smallmatrix} 1 & 0 \\ 0 & 0 \\0 & 1 \end{smallmatrix}\right)} \& F^3 \ar{r}{\left(\begin{smallmatrix} 1 & 0 & 0 \\ 0 & 1 & 0 \\ 0 & 0 & 1 \end{smallmatrix}\right)}\& F^3\ar{r}{\left(\begin{smallmatrix} 1 & 0 & 1\\ 0 & 1 & 1 \end{smallmatrix}\right)} \& F^2 \\
0\ar[u]\ar[r] \& 0\ar[u]\ar[r] \& F\ar[u]\ar{r}[swap]{\left(\begin{smallmatrix} 1 \\ 0 \end{smallmatrix}\right)} \ar{u}[swap]{\left(\begin{smallmatrix}0 \\ 1 \\ 0 \end{smallmatrix}\right)} \& F^2\ar{r}[swap]{\left(\begin{smallmatrix} 1 & 0\\ 0 & 1\end{smallmatrix}\right)}\ar{u}[swap]{\left(\begin{smallmatrix} 0 & 1 \\ 1 & 1 \\ 0 & 0 \end{smallmatrix}\right)} \& F^2\ar{u}[swap]{\left(\begin{smallmatrix} 0 & 1 \\ 1 & 1 \end{smallmatrix}\right)}
\end{tikzcd}.
\end{equation}
is obtained by applying $\operatorname{H}_1(-;F)_*$ to the bifiltration in Figure~\ref{fig.bifil}. We shall consider the generalized setting of arbitrary fields and the ring of integers in Section~\ref{sec.higherhom}.

\subsection{Homology in degree 0}
When working with homology in degree $0$, many natural bifiltrations in TDA have the additional property that in one parameter direction the linear maps in homology are epimorphisms; see the introduction of \cite{bauer2020cotorsion} for a discussion on multiparameter clustering. One may suspect that such an assumption will limit the types of indecomposables that can arise. For instance, it is not hard to show that representations with epimorphisms in one direction and monomorphisms in the other decompose into summands in a way similar to the totally ordered case. The following theorem shows that the reduction in complexity is rather limited. 
\begin{theorem}[\cite{bauer2020cotorsion}]
Linear representations of $A_{m_1}\otimes A_{m_2}$ that are epimorphic in the horizontal direction correspond (up to a finitely classified, completely explicit list of direct summands) to representations of $A_{m_1}\otimes A_{m_2-1}$. 
\end{theorem}  
In this paper we focus on a related property that applies to all diagrams of topological spaces. Namely, representations arising from homology in degree 0 factor through representations in the category of sets. That is, they are of the form $\funct{free} \circ S$, where $S$ is a commutative diagram of sets and set maps, and $\funct{free} \colon \Set \to \VectF$ denotes the functor which associates to a set $T$ the free $F$-vector space with basis $T$. That linear representations arising from a diagram of spaces by applying $\operatorname{H}_0$ are particularly well-behaved was first observed in \cite{chacholski2017combinatorial}, and more recently in \cites{brodzki2020complexity,cai2021elder,bauer2020cotorsion}. As an example, applying $\operatorname{H}_0(-;F)_*$ to the bifiltration in Figure~\ref{fig.bifil} yields the following linear representation
\[
\begin{tikzcd}[column sep=15mm, ampersand replacement=\&]
F \ar{r}{\left(\begin{smallmatrix} 0 \\ 1 \end{smallmatrix}\right)} \& F^2\ar{r}{\left(\begin{smallmatrix} 1 & 0 \\ 0 & 1 \end{smallmatrix}\right)} \& F^2 \ar{r}{\left(\begin{smallmatrix} 1 & 0 \\ 0 & 1\end{smallmatrix}\right)}\& F^2\ar{r}{\left(\begin{smallmatrix} 1 & 1 \end{smallmatrix}\right)} \& F\\
0\ar[u]\ar[r] \& 0\ar[u]\ar[r] \& F\ar[u]\ar{r}[swap]{\left(\begin{smallmatrix} 1 \\ 0 \end{smallmatrix}\right)} \ar{u}[swap]{\left(\begin{smallmatrix}1 \\ 0 \end{smallmatrix}\right)} \& F^2\ar[u]\ar{r}[swap]{\left(\begin{smallmatrix} 1 & 0\\ 0 & 1\end{smallmatrix}\right)}\ar{u}[swap]{\left(\begin{smallmatrix} 1 & 0 \\ 0 & 1 \end{smallmatrix}\right)} \& F^2\ar{u}[swap]{\left(\begin{smallmatrix}  1 & 1 \end{smallmatrix}\right)}
\end{tikzcd}.
\]

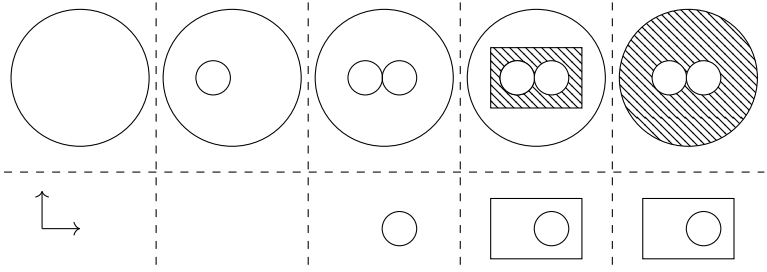
\begin{figure}
\centering
\begin{tikzpicture}

\draw[->] (-0.5, -2) -- (-0.5, -1.5);
\draw[->] (-0.5, -2) -- (-0, -2);

\draw[dashed] (-1,-1.25) -- (9, -1.25);
\draw[dashed] (1,1) -- (1, -2.5);
\draw[dashed] (3,1) -- (3, -2.5);
\draw[dashed] (5,1) -- (5, -2.5);
\draw[dashed] (7,1) -- (7, -2.5);

\draw[black] (0,0) circle (6ex); 
\draw[black] (2,0) circle (6ex); 
\draw[black] (4,0) circle (6ex); 
\draw[black] (6,0) circle (6ex); 
\draw[black, pattern=north west lines] (8,0) circle (6ex); 

\draw[black] (1.75, 0) circle (1.5ex);
\draw[black] (3.75, 0) circle (1.5ex);

\draw[black] (5.75, 0) circle (1.5ex);

\draw[pattern=north west lines] (5.4,-0.4) -- (6.6,-0.4) -- (6.6,0.4) -- (5.4,0.4) -- cycle;
\draw[black] (5.4,-2.4) -- (6.6,-2.4) -- (6.6,-1.6) -- (5.4,-1.6) -- cycle;

\draw[black] (7.4,-2.4) -- (8.6,-2.4) -- (8.6,-1.6) -- (7.4,-1.6) -- cycle;

\draw[black, fill=white] (5.75, 0) circle (1.5ex);

\draw[black, fill=white] (7.75, 0) circle (1.5ex);

\draw[black] (4.2, 0) circle (1.5ex);
\draw[black, fill=white] (6.2, 0) circle (1.5ex);
\draw[black, fill=white] (8.2, 0) circle (1.5ex);

\draw[black] (4.2, -2) circle (1.5ex);
\draw[black] (6.2, -2) circle (1.5ex);
\draw[black] (8.2, -2) circle (1.5ex);
\end{tikzpicture}
\caption{A bifiltration of a topological space.}
\label{fig.bifil}
\end{figure}
Observe that every column in the matrices of the structure maps contains exactly one entry equal to $1$, and is zero otherwise. 
This property characterizes precisely the essential image of the functor
\[
\funct{free}_* = \funct{free} \circ (-).
\]
Hence, such a basis exists not only for the representation shown above, but for any representation arising from a diagram of spaces via
$\operatorname{H}_0(-;F)_*$, since this functor factors through $\funct{free}_*$ (Proposition~\ref{prop.H0Free}). In particular, one easily convinces oneself that the representation $R$ from \eqref{eq.intro-comm} cannot be isomorphic to $\operatorname{H}_0(D;F)_*$ for any bifiltration $D$. What is less obvious is that $R$ is not even the summand of any $\operatorname{H}_0(D;F)_*$.

It is thus natural to wonder which representations ``come from homology in degree 0''. That is, we say that a representation $R \colon A_{m_1}\otimes A_{m_2} \to \VectF$ is \emph{additively $\operatorname{H}_0$-realizable} if $R$ is a summand of $\operatorname{H}_0(D;F)_*$ for some diagram of topological spaces $D \colon A_{m_1}\otimes A_{m_2} \to \Top$.

\begin{question} 
What are the additively $\operatorname{H}_0$-realizable representations $R$? 
\end{question}

While the above characterization of the image of $\operatorname{H}_0(-;F)_*$ (equivalently, of the free functor) is straightforward, answering this question appears to be very difficult in general.
As an example of how subtle being additively $\operatorname{H}_0$-realizable is, we show (Theorem~\ref{thm.Ainf}, together with Proposition~\ref{prop.Lan_is_realizable}) that if $F$ is a field of finite characteristic, then the following indecomposable representation of $A_3\otimes A_3$ is additively $\operatorname{H}_0$-realizable  if and only if $m$ is a power of the characteristic of $F$,
\[
R = \begin{tikzcd}[ampersand replacement = \&]
F^m \ar[r, "I_m"] \& F^m \ar[r]\& 0 \\
F^m \ar[u, "I_m"] \ar[r, "{\left[ \begin{smallmatrix} I_m \\ 0 \end{smallmatrix} \right]}"] \& F^{2m}\ar[u,swap, "{\left[ \begin{smallmatrix} I_m & J_m(1) \end{smallmatrix} \right]}"] \ar[r, "{\left[ \begin{smallmatrix} I_m & I_m \end{smallmatrix} \right]}"] \& F^m \ar[u] \\
0\ar[u]\ar[r] \& F^m \ar[u, "{\left[ \begin{smallmatrix} 0 \\ I_m \end{smallmatrix} \right]}"] \ar[r, "I_m"]  \& F^m\ar[u, "I_m"]\end{tikzcd}
\qquad \text{where}\qquad J_m(1) =
\begin{bmatrix}
 1 & 1 & 0& \cdots & 0 & 0 \\
 0 & 1 & 1 &\cdots & 0 & 0 \\
\vdots & \vdots & \vdots & \ddots & \vdots & \vdots \\
  0 & 0 & 0 & \cdots & 1 &  1\\
  0 & 0 & 0 & \cdots & 0 & 1
  \end{bmatrix}.
\]
While a complete understanding of the additively $\operatorname{H}_0$-realizable representations of commutative grids seems challenging, we do provide partial results in this setting, along with an algorithm for determining if a given $R$ is additively  $\operatorname{H}_0$-realizable. 


\subsection{Main question}
Before giving an overview of our results, we shall recast the above question in a more general framework. For a small category $X$ and $C$ any category, let $\rep(X, C)$ denote the category of functors from $X$ to $C$. The category of topological spaces is denoted by $\Top$, and $\VectF$ denotes the category of vector spaces over the field $F$. We let $\Set$ and $\pSet$ denote the categories of sets and pointed sets, respectively. An \emph{indicator representation} is a representation \( R \) for which there is a subset \( S \subseteq \operatorname{Obj}(X) \) such that
\[ R(x) = \begin{cases} F & x \in S \\ 0 & x \not\in S \end{cases}\qquad  \text{ and }\qquad  R(x \to y) = \begin{cases} {\rm id}_F & x,y \in S \\ 0 & \text{otherwise} \end{cases}. \]

\begin{question} \label{questions}
Given a small category $X$, what is the \emph{additive image} (the closure of the image under direct sums and summands) of the functor
\[ \operatorname{H}_0(-;F)_* \colon \rep(X, \Top) \to \rep(X, \VectF)? \]
In particular we may ask about the following three extreme cases:
\begin{enumerate}
\item Are all indecomposables in the additive image of \( \operatorname{H}_0(-;F)_* \) indicator representations?
\item Are there only finitely many indecomposables in the additive image of \( \operatorname{H}_0(-;F)_* \)?
\item Are all indecomposables in the additive image of \( \operatorname{H}_0(-;F)_* \)? That is, is \( \operatorname{H}_0(-;F)_* \) \emph{additively surjective}?
\end{enumerate}

\end{question}

Our answers to these questions will --- unsurprisingly --- depend on the category \( X \), and sometimes also on the field $F$. More surprisingly, and in contrast to the situation for linear representations, they also depend heavily on the orientation of the arrows in \( X \). For instance, the extended Dynkin quiver $\tilde D_4$,
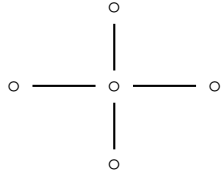
\[
\begin{tikzcd}
& \circ & \\
\circ \ar[r,-,thick] & \circ \ar[u, -, thick] \ar[d, -, thick] & \ar[l, -, thick] \circ \\
& \circ & 
\end{tikzcd}
\]
 is of tame representation type and thus has infinite families of indecomposables over any field. We prove in Section~\ref{subsect:only_thin} that if all the arrows are oriented inwards, then any indecomposable representation in the additive image has vector space dimension at most 1 over any vertex. By contrast, we prove in Section~\ref{subsec.d4} that if all the arrows are oriented outwards and $F$ is an algebraic extension of a finite field, then every representation is in the additive image. This conclusion does not hold for arbitrary infinite fields, like the reals, as shown in Example~\ref{ex:D4-trans}. 

\subsection{Overview and Structure}
In Section~\ref{sec.basic} we prove the following simple proposition, which allows us to work solely with (pointed) sets. 
\begin{resultx}[Propositions~\ref{prop.H0Free} and~\ref{prop.same_add_image}]
For a small category \( X \), the additive images of the functors 
\begin{align*}
\operatorname{H}_0(-;F)_* & \colon \rep(X, \Top) \to \rep(X, \VectF) \text{,} \\
\funct{free}_* & \colon \rep(X, \Set) \to \rep(X, \VectF) \text{, and} \\
\funct{free}^\mathrm{pt}_* & \colon \rep(X, \pSet) \to \rep(X, \VectF)
\end{align*}
coincide, where the functors in lines 2 and 3 turn each set (resp.\ pointed set) appearing in a representation into a vector space. See the beginning of Section~\ref{sec.basic} and Section~\ref{subsect.pointed_sets}, respectively, for formal definitions.
\end{resultx}
We therefore say that a representation in $\rep(X, \VectF)$ is (additively) $\Set$-realizable if it is in the (additive) essential image of any of these functors. 

In Section~\ref{sec.basic}, we also prove Theorem~\ref{thm.fieldext} which states that $R\in \rep(X, \VectF)$ is additively $\Set$-realizable if and only if $R\otimes_F F'$ is additively $\Set$-realizable, where $F'$ is any field extension of $F$. In Proposition~\ref{prop.repinfsets}, we shall also see that it suffices to consider the category of finite sets ($\set$) if $X$  is finite and the vector spaces are finite-dimensional ($\vectF$).

In Section~\ref{sec.algo} we leverage the fact that the free functor $\funct{free} \colon \Set\to \VectF$ is a left adjoint to construct an algorithm for determining if a given representation is additively $\Set$-realizable, given that the field is finite. In the following statement $|X|$ denotes the total number of morphisms in the category $X$, and $|F|$ is the number of elements in $F$. Both the category and the field are assumed to be finite.
\begin{resultx}[Corollary~\ref{cor.algo}]
\label{resX.gaussian}
Using Gaussian elimination, one can determine if $R\in \rep(X,\vectF)$ is additively $\Set$-realizable in $\mathcal{O}\left(m^3\cdot |X|^3\cdot |F|^{3m} \right)$ arithmetic operations, where $m = \max_{i\in \operatorname{Obj}(X)} \dim R(i)$.
\end{resultx}
If \( X \) is a poset, \( |X| \) can be replaced with the number of arrows in its Hasse diagram. Similarly, if \( X \) is a quiver, \( |X| \) can be chosen to be the number of arrows in \( X \); see Remark~\ref{remark.minimalarrows}.  A GAP/QPA \cite{GAP4,qpa} implementation of this algorithm is available online\footnote{Github repository: \url{https://github.com/mbotnan/set-realizable/}.}. 

In Section~\ref{sec:scalarsfiniteorder}, we show that if a representation $R$ has a basis where each basis vector is mapped to a scalar multiple of another basis vector, and the scalar is a root of unity, then $R$ is additively $\Set$-realizable. We establish this result using two distinct approaches: first, an explicit method involving matrices in Section~\ref{sec.vandermonde}, and second, as a corollary of a broader result on group actions in Section~\ref{sec.groupac}. The latter approach also yields a slight improvement in the running time of Result~\ref{resX.gaussian}, discussed in Section~\ref{sec:gset-algorithm}.

In Section~\ref{sec:transfer}, we investigate how $\Set$-realizability changes when the category $X$ is modified. Specifically, we consider cases of a full subcategory (Section~\ref{sec:fullsubcat}), a full subcategory lacking only one terminal object (Section~\ref{sec:terminalobj}), edge contractions (Section~\ref{sec.edge_contractions}), and the case of a quiver obtained from another by reversing the arrows at a source with at most two outgoing arrows (Section~\ref{sect.APR}). Notably, we show that if the indexing category contains a terminal vertex $t$ and $R$ is indecomposable and additively $\Set$-realizable, then $\dim R(t) \leq 1$ (Corollary~\ref{prop.colim_1d}). This result provides an explanation for why Eq.~\eqref{eq.intro-comm} does not belong to the additive image. Building on this, we prove the following.

\begin{resultx}[Corollary~\ref{cor.treequiver}]
Let \( Q \) be a tree quiver with a single sink $s$. Then there are only finitely many additively $\Set$-realizable indecomposable representations.
\end{resultx}
In Section~\ref{subsect:only_thin}, we give a complete answer to Question~\ref{questions} (1).
\begin{resultx}[Theorem~\ref{thm.onlythin}]
Let \( X \) be a finite category. Then all indecomposable additively \( \Set \)-realizable representations are indicator representations if and only if \( X \) is equivalent to the poset category of a poset that is a disjoint union of
\begin{enumerate}
\item posets with a Hasse diagram of type \( A \),
\item posets with a single maximal element, such that the poset without the unique maximal element is a disjoint union of posets with a Hasse diagram of type \( A \) with unique minimal elements.
\end{enumerate}
\end{resultx}
In Section~\ref{sec.quivers} we shift our focus to quiver representations. We answer all parts of Question~\ref{questions} completely for (extended) Dynkin types $A$, $D$ and $\tilde{A}$ for $F$ any field, and for the extended Dynkin type $\tilde{D}_4$ when $F$ is an algebraic extension of a finite field. In the extended case we make the assumption that vector spaces are finite-dimensional. 

Using computer algebra software, we answer Question~\ref{questions} for Dynkin type $E$ for fields of characteristic 2. 
\begin{resultx}[Proposition~\ref{prop.D4tilde_via_A3tilde} and Theorem~\ref{thm:D4_three_outgoing}]
Let $Q$ be a quiver of type $\tilde{D}_4$ and $F$ an algebraic extension of a finite field. Then the functor $\funct{free}_* \colon \rep(Q, \Set)\to \rep(Q, \vectF)$ is additively surjective if and only if $Q$ has at least three outgoing arrows. 
\end{resultx}

Building on our results for extended Dynkin quivers, we return to our original problem of grids in Section~\ref{sec.finitegrids}, and prove the following:

\begin{resultx}[Corollary~\ref{cor:finite_grods_of_infinite_type} and Observation~\ref{obs:repfinite_grids}]
For a finite grid and any field $F$, the following are equivalent:
\begin{enumerate}
\item There are infinitely many indecomposable representations.
\item There are infinitely many additively \( \Set \)-realizable indecomposable representations.
\item There are infinitely many indecomposable representations which are not additively \( \Set \)-realizable.
\end{enumerate}
Furthermore, if the grid is of finite representation type, then all representations are in the essential image of \( \funct{free}_* \).
\end{resultx}

To illustrate that this result is surprising, let us compare it to the situation for quivers: It follows from the results in Section~\ref{sec.quivers} that corresponding statement for quivers fails in both possible ways: There are quivers with infinitely many indecomposable representations but only finitely many additively \( \Set \)-realizable ones, and there are quivers with infinitely many indecomposable representations all of which are additively \( \Set \)-realizable.

In Section~\ref{sec.higherhom}, we complement the discussion from Section~\ref{sec:scalarsfiniteorder} by showing that transcendental numbers appearing in a representation obstruct additive $\Set$-realizability. Furthermore, this obstruction holds even for additive realizability as higher homologies. Our main theorem is as follows:
\begin{resultx}[Theorem~\ref{thm.char_of_Z_realizable}]
Let \( F \) be a field, \( X \) a finite category, and \( R \in \rep(X, \vectF) \). If, \( R \) is in the additive image of 
\[ \operatorname{H}_n(-;F)_* \colon \rep(X, \Top) \to \rep(X, \vectF), \] then \( R \) admits a basis such that the matrix representations of the linear maps only contain elements that are algebraic over the prime field of \( F \). Conversely, if \( R \) admits such a basis and $n\geq 1$, then $R$ is in the additive image of $ \operatorname{H}_n(-;F)_*$. 
\end{resultx}
The paper concludes with a discussion, and paths for future research.

\section{Related Work}

This paper builds on our previous work~\cite{bauer2020cotorsion}, where we investigated representations with additional algebraic properties satisfied by those arising from $\operatorname{H}_0(-;F)_*$. Since the representation theory of such representations appears intractable, our focus here is on classifying the precise representations in the (additive) image of $\operatorname{H}_0(-;F)_*$. 

In the context of topological data analysis, two works have addressed related questions. Brodzki, Burfitt, and Pirashvili~\cite{brodzki2020complexity} studied the essential image of $\operatorname{H}_0(-;F)_*$, obtaining similar results for posets to Proposition~\ref{prop.same_add_image} and Proposition~\ref{prop.graphs} in this paper; see \cite[Theorem 4.4, Theorem 5.9]{brodzki2020complexity}. Another recent work by Bindua, Brüstle, and Scoccola~\cite{bindua2024decomposing}, developed in parallel to ours, focuses on the additive image of $\operatorname{H}_0(-;F)_*$ for indexing categories that are rooted trees. They establish that only finitely many indecomposable representations exist in the additive image in this case, leading to our Corollary~\ref{cor.treequiver}. Their approach, complementary to ours, is tailored to rooted trees and leverages Kinser's work~\cite{kinser2010rank} to describe all indecomposables in the additive image~\cite[Corollary C]{bindua2024decomposing}. Moreover, they present a quadratic-time algorithm for decomposing representations obtained by applying $\operatorname{H}_0$ to a filtration of a graph $G$, where the input size is the number of edges and vertices in $G$~\cite[Theorem D]{bindua2024decomposing}.

For higher-degree homology, Carlsson and Zomorodian~\cite{carlsson2007theory} showed that for prime fields $F$ the functor $\operatorname{H}_n(-;F)_*$ is essentially surjective. This result was later extended to $\Z$ coefficients by Harrington, Otter, Schenck, and Tillmann~\cite{harrington2019stratifying}. Our work is the first to address the case where $F$ is not necessarily prime.

\section{First properties}
\label{sec.basic}

\subsection{Setting the stage}

Throughout this paper, we will be considering representations of small categories. The two classes of examples we most often have in mind are the following:

The \emph{poset cateogory} of a (often finite) poset \( X \), that is the category whose objects are the elements of \( X \), and whose morphisms are the relations \( x_1 \leq x_2 \in X \).

The \emph{path category} of a (typically finite) quiver \( Q \). (Quiver is another word for directed graph.) That is the category where the objects are the vertices of $Q$, and the morphisms are directed paths in $Q$.

By abuse of notation we will mostly identify a poset with its poset category and a quiver with its path category. 

For a small category $X$ and a category $C$, we define a \emph{representation of $X$ in $C$} to be a functor $R\colon X\to C$, and we denote the corresponding functor category by $\rep(X,C)$. For most of this paper, $C$ is the category of topological spaces, of (pointed) sets, or  of (finite-dimensional) vector spaces. We adapt the following notation: 
\begin{enumerate}
\item $\set\subseteq \Set$ = (finite) sets, $\pset \subseteq \pSet$ = (finite) pointed sets,
\item $\vectF\subseteq\VectF$ = (finite-dimensional) $F$-vector spaces.
\end{enumerate}
We shall repeatedly make use of the fact that the free functor
$\funct{free} \colon \Set \to \VectF$
is left adjoint to the forgetful functor $\funct{forget} \colon \VectF \to \Set$.

For a functor \( \funct{F} \colon C_1 \to C_2 \), and any small category \( X \), we get an associated functor \( \funct{F}_* \colon \rep(X, C_1) \to \rep(X, C_2) \) given by the composition \( \funct{F}_*(R) = \funct{F} \circ R \). In particular we will be concerned with the functors \( \funct{free}_* \) and \( \funct{forget}_* \).

For a functor \( T \colon A \to B \), the \emph{essential image} is the collection of all objects \( Y \) in \( B \) such that \( Y \cong T(Z) \) for some \( Z \in \operatorname{Obj}(A) \). When \( B \) is additive (e.g., vector spaces or abelian groups), the \emph{additive image} is the collection of all objects \( Y \) such that \( Y \oplus Y' \cong T(Z) \) for some \( Z \in \operatorname{Obj}(A) \) and \( Y' \in \operatorname{Obj}(B) \), together with all finite direct sums of such objects \(Y\). 

\begin{remark}
\label{rem.addimage}
Note that if \(A\) admits finite coproducts and \(T\) preserves them, then closure under finite direct sums is automatic.

In Corollary~\ref{cor.adj_summand} we will see that \( T \) being left adjoint also makes closure under finite direct sums automatic.

In the situations considered throughout this paper, both conditions will always be true, so the main point of interest is closure under direct summands.
\end{remark}

In this preliminary subsection, we point out that the essential (additive) images of \( \operatorname{H}_0(-; F)_* \) and \( \funct{free}_* \) coincide, discuss two immediate properties of the latter, and finally point out that for \( \operatorname{H}_0(-;F)_* \) it suffices to consider graphs rather than arbitrary topological spaces.

\begin{proposition}
\label{prop.H0Free}
Let $X$ be a small category. The essential image of the functor \[ \operatorname{H}_0(-; F)_* \colon \rep(X, \Top ) \to \rep(X, \VectF) \] is the same as the essential image of the functor \[ \funct{free}_* \colon \rep(X, \Set ) \to \rep(X, \VectF). \]
In particular, also the additive images of ${H}_0(-; F)_*$ and $\funct{free}_*$ coincide.
\end{proposition}
\begin{proof}
This follows immediately from the fact that the two functors factor through one another. In the one direction, consider the functor taking a set to the corresponding discrete topological space. In the other direction, consider the functor $\pi_0$ taking any topological space to the set of its path-connected components. 
\end{proof}

First, we observe that the essential image of $\funct{free}_*$ can be easily described.
\begin{proposition}
\label{prop:free-mult-basis}
Let $X$ be a small category and \( R \in \rep(X, \VectF ) \). Then \( R \) lies in the essential image of \( \funct{free}_* \) if and only if \( R \) has a basis such that any arrow maps a basis vector to a basis vector.
\end{proposition}
\begin{proof}
This is immediate from the definition of $\funct{free}_*$.
\end{proof}

Second, if we are only interested in finite-dimensional linear representations of finite categories, then it suffices to consider finite sets by the following result.

\begin{proposition} 
\label{prop.repinfsets}
Let $X$ be a finite category and \( R \in \rep(X, \vectF) \). Then $R$ is in the additive image of the functor \( \funct{free}_* \colon \rep(X, \Set) \to \rep(Q, \VectF) \) if and only if it is in the additive image of the functor \( \funct{free}_* \colon \rep(X, \set) \to \rep(Q, \vectF) \).
\end{proposition}

\begin{proof}
Let \( S \in \rep(X, \Set ) \) be such that \( R \) is a direct summand of \( \funct{free} \circ S \). We need to show that there is \( S' \in \rep(X, \set ) \) such that \( R \) is also a direct summand of \( \funct{free} \circ S' \). Let \( \varphi \colon \funct{free} \circ S \longleftrightarrow R \colon \psi \) be a pair of morphisms such that \( \varphi \circ \psi = \operatorname{id}_R \). For each vertex \( i \in X \), the vector space \( \Im \psi_i \subseteq \funct{free}(S_i) \) is finite-dimensional. In particular, there is a finite subset \( S^{\rm fin}_i \subseteq S_i \) such that \( \Im \psi_i \subseteq \funct{free}(S^{\rm fin}_i) \). Now we can choose \( S' \) to be the sub-representation of \( S \) generated by all the subsets \( S^{\rm fin}_i \). Since our category \( X \) is finite, \( S' \in \rep(X, \set ) \) as desired. Moreover, we have the following commutative diagram:
\[\begin{tikzcd}
R \ar[rrrr, bend left, "\psi"] \ar[rr, "\psi"]& & \funct{free} \circ S' \arrow[hookrightarrow]{rr}& & \funct{free} \circ S \ar[llll, bend left, "\phi"]
\end{tikzcd}
\]
It follows that \( R \) is also a direct summand of \( \funct{free} \circ S' \).
\end{proof}
We conclude this subsection with the following simple observation.
\begin{proposition}
\label{prop.graphs}
Let $X$ be a finite poset, and assume that $R\in \rep(X, \vectF)$ is in the additive image of $\operatorname{H}_0(-;F)_*$. Then there exists a diagram $G\in\rep(X,\Top)$ of finite 1-dimensional simplicial complexes (graphs) such that $G_p\subseteq G_q$ for $p\leq q$, and such that $R$ is a direct summand of $\operatorname{H}_0(G; F)_*$.
\end{proposition}
\begin{proof}
From Proposition~\ref{prop.H0Free} and Proposition~\ref{prop.repinfsets}, there exists an $S\in \rep(X, \set)$ such that $R$ is a direct summand of $\funct{free} \circ S \in \rep(X, \vectF)$. Furthermore, by considering $S$ as diagram of 0-dimensional simplicial complexes and simplicial maps, we have that $\funct{free} \circ S \cong H_0(S;F)_*$. Let $G(p)$ be the simplicial complex with vertices given by the disjoint union $\bigsqcup_{r\leq p} S(p)$, and with an edge connecting $(i,q)$ and $(j,q')$ if $S(q\to q')(i)=j$. Since there exists a unique morphism from $r$ to $p$, $(i,r)$ is in the same connected component of $G(p)$ as precisely one element $(j,p)$, namely $(S(q\to q')(i), p)$. It follows that the pointwise inclusion $S(p) \hookrightarrow G(p)$ mapping $i$ to $(i, p)$ induces an isomorphism $H_0(S;F)_* \cong H_0(G;F)_*$. 
\end{proof}

\subsection{The additive image over pointed sets} \label{subsect.pointed_sets}
Often it is more convenient to work with pointed sets. The next observation is that, at least when considering the additive image, this is equivalent to considering sets.

Note that there is a forgetful functor \( \funct{forget}^\mathrm{pt} \colon \Vect \to \pSet \), remembering the zero vector of the vector space as base point. This functor also has a left adjoint \( \funct{free}^\mathrm{pt} \colon \pSet \to \Vect \) given explicitly by the free vector space on the set minus its base point. As before, we will also consider the functors \( \funct{forget}^\mathrm{pt}_* \) and \( \funct{free}^\mathrm{pt}_* \) induced on categories of representations by compositions.

\begin{lemma} \label{obs.image_inclusion}
Let $X$ be a small category. Then the essential image of the functor \( \funct{free}_* \colon \rep(X, \Set ) \to \rep(X, \VectF) \) is contained in the essential image of \( \funct{free}^\mathrm{pt}_* \colon \rep(X, \pSet ) \to \rep(X, \VectF)\).
\end{lemma}

For the proof of this lemma, we need the following observation:

\begin{observation}
The forgetful functor \( \funct{forget-base} \colon \pSet \to \Set \) has a left adjoint, given by \( S \mapsto S \cup \{ * \} \), where \( * \) is the base point of the image. Indeed, a map from \( S \cup \{ * \} \) to any pointed set is given by a set map on \( S \) -- we have no choice where to send the base point \( * \).
\end{observation}

\begin{proof}[Proof of Lemma~\ref{obs.image_inclusion}]
The forgetful functor \( \VectF \to \Set \) naturally factors as follows:
\[ \begin{tikzcd}
\VectF \ar[rrrr, bend left, "\funct{forget}"] \ar[rr, "\funct{forget}^\mathrm{pt}"] & & \pSet \ar[rr, "\funct{forget-base}"] & & \Set
\end{tikzcd} \]
In particular, by uniqueness of (left) adjoints, we have a natural factorization of functors from \( \Set \) to \( \VectF \):
\[ \funct{free} = \funct{free}^\mathrm{pt} \circ (- \cup \{ * \}). \qedhere \]
\end{proof}

\begin{lemma} \label{lem.const_splits_off}
We have a natural isomorphism of functors
\[ \funct{free} \circ \funct{forget-base} \iso \funct{free}^\mathrm{pt} \oplus \funct{const}_F \colon \pSet \to \VectF, \]
where $\funct{const}_F$ is the constant functor sending any pointed set to the field $F$ and any morphism to the identity.
\end{lemma}

\begin{proof}
For a pointed set \( S \) we have the short exact sequence
\[ \begin{tikzcd}
0 \ar[r] & F \ar[r] & \funct{free}(\funct{forget-base}(S)) \ar[r] \ar[l, bend right, dashed] & \funct{free}^\mathrm{pt}(S) \ar[r] & 0, 
\end{tikzcd} \]
where the first map sends \( 1 \) to the base point of $S$, and the second map sends the base point to $0$ and any other element of \( S \) to itself. Note that this short exact sequence is functorial in \( S \). There is a natural splitting given by sending any linear combination \( \sum_{s \in S} f_s s \) to \( \sum_{s \in S} f_s \).
\end{proof}

\begin{remark}
In the previous proof, one might first be tempted to consider the splitting given by sending the base point to 1 and all other elements of \( S \) to 0. However, that splitting is \emph{not} natural. (The problem is that morphisms may send other elements of a pointed set to the base point.)
\end{remark}

\begin{proposition} \label{prop.same_add_image}
Let $X$ be a small category and \( R \in \rep( X, \pSet ) \). Then
\[ \funct{free} \circ \funct{forget-base} \circ R  \iso \funct{free}^\mathrm{pt} \circ R \oplus F_X. \]
In particular, the additive images of \( \funct{free}_* \) and \( \funct{free}^\mathrm{pt}_* \) coincide. 
\end{proposition}

\begin{proof}
The first statement is an immediate consequence of Lemma~\ref{lem.const_splits_off}. It follows that any representation in the image of \( \funct{free}^\mathrm{pt}_* \) is a direct summand of a representation in the image of \( \funct{free}_* \). The equality of the additive images now follows together with Lemma~\ref{obs.image_inclusion}.
\end{proof}
As with $\funct{free}_*$, the essential image of $\funct{free}^\mathrm{pt}_*$ can also be easily described. 
\begin{proposition}
\label{prop:mult-basis}
Let $X$ be a small category and \( R \in \rep(X, \VectF) \). Then \( R \) lies in the essential image of \( \funct{free}^\mathrm{pt}_* \) if and only if \( R \) has a \emph{coherent basis}, that is, a basis such that any structure morphism maps any basis vector to a basis vector or to zero.
\end{proposition}
\begin{proof}Immediate from the definition of $\funct{free}^\mathrm{pt}_*$.
\end{proof}

\begin{corollary}\label{cor.const_essimg}
Any indicator representation is in the essential image of $\funct{free}^\mathrm{pt}_*$.
\end{corollary}

\begin{remark}
Propositions~\ref{prop.same_add_image} and \ref{prop:mult-basis} first appeared in  \cite{brodzki2020complexity} (Theorem 5.8 and Lemma 4.3, respectively) for finite-dimensional representations of finite partially ordered sets. 
\end{remark}

\begin{remark} 
Note that the essential image of $\funct{free}^\mathrm{pt}_*$ is larger than that of $\funct{free}_*$ but in many cases smaller than the additive image.
For example, if 
\end{remark}

\subsection{Field extensions}
One may ask how changing the field we consider representations over affects realizability. The following result shows that field extensions do not change the realizability of a given representation.

\begin{theorem}\label{thm.fieldext}
Let $R\in \rep(X, \vectF)$ where $X$ is a finite category, and let $F'$ be any field extension of $F$. Then $R$ is additively $\Set$-realizable if and only if $R\otimes_F F' \in \rep(X, \vect_{F'})$ is additively $\Set$-realizable.
\end{theorem}
\begin{proof}
Observe that we have the following commutative diagram
\begin{equation}
\label{eq.free}
\begin{tikzcd}
\rep(X,\set) \ar[r, "\funct{free}_{F'}"]\ar[dr, swap, "\funct{free}_F"] & \rep(X,\vect_{F'})\\
& \rep(X,\vectF) \ar[u, swap, "-\otimes_F F'"]
\end{tikzcd}
\end{equation}
Note that \( - \otimes_F F' \) commutes with direct sums. Hence, if $R$ is a direct summand of $\funct{free}_F \circ S$ for some $S\in \rep(X, \set)$, then \( R \otimes_F F' \) is a direct summand of \( (\funct{free}_F \circ S) \otimes_F F' \). By the commutative triangle above, the latter is isomorphic to \( \funct{free}_{F'} \circ S \).
We conclude that $R\otimes_F F'$ is additively $\Set$-realizable. 

To prove the converse, let $B$ denote any basis for $F'$ as a vector space over $F$. Then, for any $R\in \rep(X, \vectF)$, we have 
\begin{equation}
\label{eq.fieldext}
\rho(R\otimes_F F') \cong \bigoplus_{b\in B}R\in \rep(X, \VectF),
\end{equation}
where $\rho$ denotes restriction of scalars along the inclusion $F\hookrightarrow F'$. Assume that $R \otimes_F F'$ is a direct summand of $\funct{free}_{F'} \circ S$ for some $S\in (X, \set)$. Then, since \( \rho \) is an additive functor, \( \bigoplus_{b\in B} R \cong \rho(R\otimes_F F') \) is a direct summand of
\[ \rho(\funct{free}_{F'} \circ S) \overset{(\ref{eq.free})}{\cong} \rho(\funct{free}_F \circ S\otimes_F F') \overset{(\ref{eq.fieldext})}\cong \bigoplus_{b\in B} \funct{free}_F \circ S \cong \funct{free}_F \circ \left(\bigsqcup_{b\in B} S\right). \]
We conclude that $\bigoplus_{b\in B} R$ --- and therefore $R$ --- is in the additive image of $(\funct{free}_F)_* \colon \rep(X, \Set)\to \rep(X, \VectF)$. Since $X$ was assumed to be finite, the result follows from Proposition~\ref{prop.repinfsets}.
\end{proof}

\section{Determining membership in the additive image}
\label{sec.algo}
In this section we describe an algorithm for determining whether a linear representation is additively \(\Set\)-realizable. We begin with the following lemma. Recall that a morphism \(e \colon X \to Y\) is a \emph{split epimorphism} if there exists \(s \colon Y \to X\) with \(e \circ s = \mathrm{id}_Y\), and that an object \(X\) is in the \emph{epimorphic image} of \(Y\) if there exists an epimorphism \(s \colon Y \to X\). Note that every split epimorphism is in particular an epimorphism.

\begin{lemma} \label{lemma.split_epi_counit}
Let \( \mathtt F \dashv \mathtt G \) be an adjoint pair of functors, between two categories \( \mathcal{C} \) and \( \mathcal{D} \). An object \( D \in \mathcal{D} \) is the (split) epimorphic image of some object in the image of \( \mathtt{F} \) if and only the counit \( \mathtt F \mathtt G D \to D \) is a (split) epimorphism.
\end{lemma}

\begin{proof}
The ``if''-part is clear.

So assume there is a (split) epimorphism \( e \colon \mathtt{F}C \to D \). By naturality of the counit the following square commutes.
\[ \begin{tikzcd}
\mathtt{F} \mathtt{G} \mathtt{F} C \ar[r,"\mathtt{F}\mathtt{G}e"] \ar[d,swap,"\operatorname{counit}_{\mathtt{F}C}"] & \mathtt{F} \mathtt{G} D \ar[d,"\operatorname{counit}_{D}"]  \\
\mathtt{F} C  \ar[r,"e"] & D
\end{tikzcd} \]
Note that the counit on \( \mathtt{F}C \) is always split epi (a splitting being given by \( \mathtt{F} \) of the unit on \( C \)). It follows that the composition from the left upper to the right lower corner of the diagram is (split) epi, and hence so is the counit on \( D \).
\end{proof}

The next lemma shows that when the second category is additive then the condition described in Lemma~\ref{lemma.split_epi_counit} propergates to sums; see also Remark~\ref{rem.addimage}.

\begin{lemma} \label{lemma:sums_ok_for_adj}
Let \( \mathtt F \dashv \mathtt G \) be an adjoint pair of functors, between two categories \( \mathcal{C} \) and \( \mathcal{A} \), where \( \mathcal{A} \) is additive. Let \( A_1 \) and \( A_2 \) be objects in \( \mathcal{A} \) such that the counit morphims \( \mathtt F \mathtt G A_i \to A_i \) are split epimorphisms. Then so is the counit morphim \( \mathtt F \mathtt G (A_1 \oplus A_2) \to A_1 \oplus A_2 \).
\end{lemma}

\begin{proof}
We consider the diagram
\[ \begin{tikzcd}[column sep=30mm,row sep=15mm,ampersand replacement=\&]
\mathtt F \mathtt G A_1 \oplus \mathtt F \mathtt G A_2 \ar[r,"{\left[ \begin{smallmatrix} \operatorname{counit}_{A_1} & 0 \\ 0 & \operatorname{counit}_{A_2} \end{smallmatrix} \right]}"] \ar[d,swap,"{\left[ \begin{smallmatrix} \mathtt F \mathtt G (A_1 \to A_1 \oplus A_2) \\ \mathtt F \mathtt G (A_2 \to A_1 \oplus A_2) \end{smallmatrix} \right]}"] \& A_1 \oplus A_2 \\
\mathtt F \mathtt G (A_1 \oplus A_2) \ar[ru,swap,"\operatorname{counit}_{A_1 \oplus A_2}"]
\end{tikzcd} \]
This diagram commutes since the counit is a natural transformation.

Our assumption implies that the top arrow is a split epimorphism. Then it follows that also the diagonal arrow needs to be a split epimorphism.
\end{proof}

Combining the above lemmas, we get the following result.

\begin{corollary} \label{cor.adj_summand}
Let \( \mathtt F \dashv \mathtt G \) be an adjoint pair of functors, between two categories \( \mathcal{C} \) and \( \mathcal{A} \), where \( \mathcal{A} \) is additive. An object \( A \in \mathcal{A} \) is in the additive image of \( \mathtt F \) if and only if the counit \( \mathtt F \mathtt G A \to A \) is a split epimorphism.

In particular, it is in the additive image only if it is a direct summand of some object in the image.
\end{corollary}

\begin{proof}
Recall that the additive image consists of all finite direct sums of summands of objects in the image. Considering first a summand of an object in the image, the claim follows from Lemma~\ref{lemma.split_epi_counit}. Now, for finite sums of such summands the claim follows from Lemma~\ref{lemma:sums_ok_for_adj}. 
\end{proof}

\begin{corollary} \label{cor.criterion_for_add_image}
Let \( X \) be a small category, and \( R \in \rep( X, \VectF) \). Then \( R \) is in the additive image of the functor \( \funct{free}_* \colon \rep( X, \Set ) \to \rep( X, \VectF) \) if and only if the natural epimorphism \( \funct{free}^\mathrm{pt} \circ \funct{forget}^\mathrm{pt} \circ R \to R \) splits.
\end{corollary}
This corollary shows that checking if a representation is additively $\Set$-realizable amounts to a (possibly large) linear algebra problem. We consider the running time in Section~\ref{subsect:runtime}.

\begin{example}\label{ex.D4}
Consider the following indecomposable representation $R$,
\[ \begin{tikzcd}[sep=10mm]
& \Z/2\Z \ar{d}{\begin{bmatrix} 1 \\ 0 \end{bmatrix}} & \\
\Z/2\Z  \ar{r}{\begin{bmatrix} 0 \\ 1 \end{bmatrix}} & \Z/2\Z\oplus \Z/2\Z  & \Z/2\Z  \ar{l}[swap]{\begin{bmatrix} 1 \\ 1 \end{bmatrix}} 
\end{tikzcd} \]
for which $\funct{free}^\mathrm{pt} \circ \funct{forget}^\mathrm{pt} \circ R$ equals
\[ \begin{tikzcd}[sep=10mm]
& \Z/2\Z \ar{d}{\begin{bmatrix} 1 \\ 0 \\ 0 \end{bmatrix}} & \\
\Z/2\Z  \ar{r}{\begin{bmatrix} 0 \\ 1 \\ 0 \end{bmatrix}} & (\Z/2\Z)^3  & \Z/2\Z \ar{l}[swap]{\begin{bmatrix} 0 \\ 0 \\ 1 \end{bmatrix}}.
\end{tikzcd} \]
This representation is a direct sum of three indicator representations. We conclude from Corollary~\ref{cor.criterion_for_add_image} that $R$ is not additively $\Set$-realizable. 
\end{example}

In Example~\ref{example:D4_by_G-set} we will see that the claim holds for any finite field, and in Example~\ref{ex.D4_inward} we give a more systematic argument for this fact.

\begin{example}
\label{ex.7star}
Consider the following representation $M$ of the 7-star quiver:
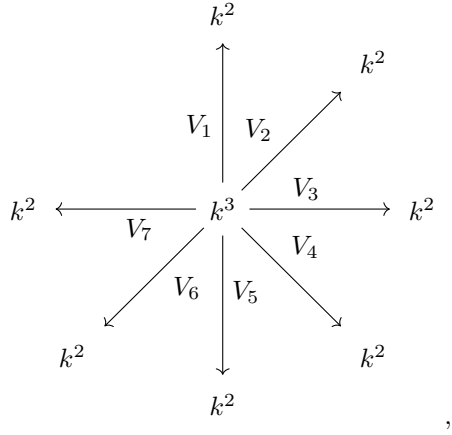
\[
\begin{tikzpicture}[scale=0.7]
\node at (0,0) (o) {$F^3$};
 \foreach \X [count=\Y]in {1,2,3,4,5,6,7}
  \draw[->,shorten <=10pt] (0,0) -- node[auto] {$V_{\X}$}(135-\Y*45:pi) 
   node[label=135-45*\Y: $F^2$]{};
\end{tikzpicture},
\]
where
\begin{align*}
V_1 = \begin{bmatrix} 1 & 1 & 0 \\ 0 & 1 & 1 \end{bmatrix} & & V_2 = \begin{bmatrix} 1 & 0 & 0 \\ 0 & 1 & 0 \end{bmatrix} & & V_3 = \begin{bmatrix} 1 & 0 & 0 \\ 0 & 1 & 1\end{bmatrix} & & V_4 = \begin{bmatrix} 1 & 0 & 1 \\ 0 & 1 & 0 \end{bmatrix}\\  
V_5 = \begin{bmatrix} 0 & 1 & 0 \\ 0 & 0 & 1 \end{bmatrix} & & V_6 = \begin{bmatrix} 1 & 0 & 0 \\ 0 & 0 & 1 \end{bmatrix} & & V_7  = \begin{bmatrix} 1 & 1 & 0 \\ 0 & 0 & 1 \end{bmatrix} & ~ &
\end{align*}
For $F=\Z/2\Z$, these matrices constitute the set of full rank $2\times 3$ matrices up to row equivalence. Note that each matrix in the equivalence class of $V_1$ has one column with more than one non-zero entry. Furthermore, if $A$ is any invertible $3\times 3$ matrix, then precisely one of the matrices $\{V_iA\}_{i=1}^7$ will be row-equivalent to $V_1$, because the matrices $\{V_i\}_{i}$ generate all row equivelence classes of full rank $2\times 3$ matrices.

We conclude that any basis of $F^3$ contains a vector that is mapped to a vector with more than one non-zero entry by one of the $V_i$, for any basis of $F^2$.
In other words, $M$ does not have a coherent basis, and, from Proposition~\ref{prop:mult-basis}, we conclude that $M$ is not in the essential image of $\funct{free}^\mathrm{pt}_*$. More interestingly, one can check using computer algebra software\footnote{\url{https://github.com/mbotnan/set-realizable/}} that the natural epimorphism $\funct{free}^\mathrm{pt} \circ \funct{forget}^\mathrm{pt} \circ M \to M$ does not split. In particular, $M$ is not additively $\Set$-realizable. We do not know of a more direct argument for coming to this conclusion. Note that the natural epimorphism \emph{does} split when working over $F=\Z/3\Z$.
\end{example}

\subsection{Running time for finite categories} \label{subsect:runtime}
Let $X$ be a finite category with objects $\operatorname{Obj}(X)$, morphisms $\operatorname{Mor}(X)$, and let $|X|$ denote the total number of morphisms in $X$. We shall assume that $F$ is finite field with $|F|$ elements.

For $R\colon X\to \vectF$ and $v\in \operatorname{Obj}(X)$, we let $d_v = \dim R(v)$ and $d = \sum_{v\in \operatorname{Obj}(X)} d_v$. Let $e\colon \funct{free}^\mathrm{pt} \circ \funct{forget}^\mathrm{pt} \circ R \to R$ denote the counit from Corollary~\ref{cor.criterion_for_add_image}.

Determining if $R$ is additively $\Set$-realizable amounts to probing the existence of an $f\colon R\to \funct{free}^\mathrm{pt} \circ \funct{forget}^\mathrm{pt} \circ R$ such that $e \circ f = \operatorname{id}_R$.
Fixing a basis for each $R(v)$ and $\funct{free}^\mathrm{pt} \circ \funct{forget}^\mathrm{pt} \circ R(v)$, this reduces to a simple (albeit large) linear algebra problem. For each $v$ and each morphism $v\to w$ we let $E_v$, $L_{(v,w)}$ and $U_{(v,w)}$ denote the matrix representations of $e(v)$, $R(v\to w)$, and $\funct{free}^\mathrm{pt} \circ \funct{forget}^\mathrm{pt} \circ R(v\to w)$, respectively. Note that these matrices are of dimensions  $d_v\times (|F|^{d_v}-1)$, $d_w\times d_v$, and $(|F|^{d_w}-1)\times (|F|^{d_v}-1)$. A right inverse to $e$ exists if and only if there exists a family $\{G_v\}_{v\in \operatorname{Obj}(X)}$ of $(|F|^{d_v}-1)\times d_v$-matrices satisfying
\begin{itemize}
\item $E_vG_v = I_{d_v} (\text{the $d_v\times d_v$ identity matrix})$ for each $v\in \operatorname{Obj}(X)$, 
\item $G_wL_{(v,w)} = U_{(v,w)}G_v$ for all (non-identity) morphisms $v\to w\in \operatorname{Mor}(X)$. 
\end{itemize} 
We conclude the following.
\begin{proposition}\label{prop.sys_eq}
Let $R\in \rep(X,\vectF)$. One can check if $R$ is additively $\Set$-realizable by solving a linear system over $F$ with
\begin{align*}
\#{\rm indeterminants} &= \sum_{v\in \operatorname{Obj}(X)}d_v(|F|^{d_v}-1)  \\
\#{\rm equations} &= \sum_{v\in \operatorname{Obj}(X)}d_v^2 + \sum_{\substack{v\to w \in \operatorname{Mor}(X)\\ \text{non-identity}}} d_v(|F|^{d_w}-1).
\end{align*}
\end{proposition}
\begin{corollary}
\label{cor.algo}
Using Gaussian elimination, one can determine if $R\in \rep(X,\vectF)$ is additively $\Set$-realizable in $\mathcal{O}\left(m^3\cdot |X|^3\cdot |F|^{3m} \right)$ arithmetic operations, where $m = \max_{v\in \operatorname{Obj}(X)} \dim R(v)$.
\end{corollary}
\begin{proof}
Observe that 
\[ \sum_{v\in \operatorname{Obj}(X)}d_v(|F|^{d_v}-1) \leq |X|\cdot m\cdot |F|^m ,\]
and 
\[\sum_{v\in \operatorname{Obj}(X)}d_v^2 + \sum_{\substack{v\to w \in \operatorname{Mor}(X)\\ \text{non-identity}}} d_v(|F|^{d_w}-1) \leq |X|(m^2+m\cdot |F|^m).\]
A system of $E$ equations in $V$ variables, can be solved by Gaussian elimination in $\mathcal{O}(EV^2)$. The result is now immediate. 
\end{proof}

An implementation of this algorithm and several examples (including Example~\ref{ex.7star}) are available in our repository\footnote{\url{https://github.com/mbotnan/set-realizable/}}.
\begin{remark}
\label{remark.minimalarrows}
It is not necessary to check the commutativity condition \( G_wL_{(v,w)} = U_{(v,w)}G_v \) for all morphisms: Given a minimal generating set $S$ for $\operatorname{Mor}(X)$ (that is, any morphism in $\operatorname{Mor}(X)$ can be written as the composition of morphisms in $S$), it suffices to check commutativity for the morphisms in \( S \); commutativity for all other morphisms then holds automatically. It follows that one may replace $|X|$ by $|S|$ in the formula of Corollary~\ref{cor.algo}.

This means that if $X$ is a quiver, it suffices to consider the arrows in $X$, and if $X$ is a poset, then one only has to check commutativity for arrows in the Hasse diagram.
\end{remark} 

\section{Scalars of finite order}
\label{sec:scalarsfiniteorder}
We observed (in Proposition~\ref{prop:mult-basis}) that a representation lies in the essential image of \( \funct{free}^\mathrm{pt}_* \) if and only if it has a coherent basis, that is a basis such that any structure map applied to any basis vector gives either a basis vector or zero.

In this section we will observe that when considering the additive image, we can relax the assumption. More precisely, we show that if any basis vector is mapped to a scalar multiple of a basis vector, and all scalars involved are roots of unity, then the representation is already additively \( \Set \)-realizable. This result holds for any field, however it is particularly powerful for finite fields, since in that case any non-zero scalar is a root of unity.

We give two independent proofs of this fact, the first being a more explicit argument using matrices and the second a more abstract argument using group actions. In the final subsection of this section, we observe that the discussion here leads to a (slightly) improved version of the algorithm in Section~\ref{sec.algo}.

In Section~\ref{subsec:Atilde} (in particular Theorem~\ref{thm.Ainf}) we will see examples showing that the assumption of scalars being roots of unity is necessary, so that the result in this section is in a sense as strong as possible.

\subsection{Removing scalars via Vandermonde matrices}
\label{sec.vandermonde}
\begin{theorem} \label{thm.isinduced2}
Let $X$ be a small category and \( R\in (X, \vectF) \), and assume that $F$ has a primitive $n$-th root of unity for some $n$. Then, $R$ lies in the additive image of $\funct{free}^\mathrm{pt}_*$ if $R$ has a basis such that any arrow maps any basis vector to a multiple $\lambda$ of a basis vector, where $\lambda$ is either an $n$-th root of unity or 0. 
\end{theorem}

Before we work towards a proof of this theorem, we give two corollaries that are given by situations where the assumptions of the theorem are automatically satisfied. Here we call a basis \emph{almost coherent} if any morphism in \( X \) maps any basis vector to a scalar multiple of a basis vector.
  
\begin{corollary} \label{cor.near_mult_basis_is_enough}
For a finite field \( F \), any \( R \in(X, \VectF) \) having an almost coherent basis is additively \( \Set \)-realizable.
\end{corollary}

Note that here we did not have to explicitly require (primitive) \( n \)-th roots of unity: Since the multiplicative group of a finite field is cyclic, it automatically has a primitive $(|F|-1)$-st root of unity, and all non-zero field elements are powers of this primitive root of unity.

\begin{corollary} \label{cor.finite_cat_is_enough}
Let \( F \) be an algebraic extension of a finite field. For a finite category \( X \), any \( R \in \rep(X, \vectF) \) having an almost coherent basis is additively \( \Set \)-realizable.
\end{corollary}

\begin{proof}
This corollary follows from the one above observing that, since \( X \) is finite, \( R \) is actually defined over a finite subfield of \( F \).
\end{proof}

Now we start working towards a proof of Theorem~\ref{thm.isinduced2}. We prepare the following lemma which we will then use for a suitable base change.

\begin{lemma} \label{lem.base_change}
Let \( \zeta \) be a primitive \( n \)-th root of unity in $F$. Then, 
\begin{enumerate}
\item the matrix \( Z = [ \zeta^{ij} ]_{i,j=0}^{n-1} \) is invertible,
\item for any \( n \)-th root of unity \( x \), the diagonal matrix \( \operatorname{diag}(1, x, x^2, \cdots, x^{n-1}) \) is conjugate via \( Z \) to a permutation matrix.
\end{enumerate}
\end{lemma}

\begin{proof}
The matrix \( Z \) is a special Vandermonde matrix, so the first point follows. (Note that by assumption all the \( \zeta^i \) with \( i \) from \( 0 \) to \( n-1 \) are pairwise different.)

For the second point, first note that \( x = \zeta^r \) for some \(r \in \{ 0, \ldots, n-1 \} \). Now we observe that
\begin{align*}
Z \operatorname{diag}(1, x, x^2, \cdots, x^{n-1}) & = [ \zeta^{ij} ]_{i,j=0}^{n-1} \operatorname{diag}(1, \zeta^r, \zeta^{2r}, \cdots, \zeta^{(n-1)r}) \\
& = [\zeta^{ij} \zeta^{jr}]_{i,j = 0}^{n-1} = [\zeta^{(i+r)j}]_{i,j = 0}^{n-1}.
\end{align*}
That is, \( Z \operatorname{diag}(1, x, x^2, \cdots, x^{n-1}) \) is obtained from \( Z \) by moving all rows down by \( r \) positions (cyclically), and consequently \( Z \operatorname{diag}(1, x, x^2, \cdots, x^{n-1}) Z^{-1} \) is the corresponding permutation matrix.
\end{proof}

\begin{proposition} \label{prop.roots->permutations}
Let $R \in \rep(X,\vectF)$. Assume \( F \) has a primitive \( n \)-th root ofunity, and that all matrix entries of \( R \) are zero or \( n \)-th roots of unity, for some \( n \). Then \( R \) is a direct summand of a representation \( \widehat R \) obtained from \( R \) by replacing any matrix entry by a permutation matrix.
\end{proposition}

\begin{proof}
Let \( R_i \) be the representation obtained from \( R \) by replacing any non-zero matrix entry by its \( i \)-th power --- note that this is in fact a representation since \( R \) has an almost coherent basis. Let \( \widehat R = \bigoplus_{i=0}^{n-1} R_i \); clearly \( R = R_1 \) is a direct summand of \( \widehat R \). The matrices for the representation \( \widehat{R} \) may be obtained by replacing in the matrix representation for \( R \) any value \( x \) by the diagonal matrix \( \operatorname{diag}(1, x, x^2, \cdots, x^{n-1}) \) (and replacing zeros by \( n \times n \) zero matrices). In particular, applying a block-wise base change as in Lemma~\ref{lem.base_change} we arrive at the claim.
\end{proof}

\begin{proof}[Proof of Theorem~\ref{thm.isinduced2}]
By Proposition~\ref{prop.roots->permutations}, \( R \) is a direct summand of the representation obtained from \( R \) be replacing all non-zero matrix entries by permutation matrices. But for these matrices there will be at most one \( 1 \) in every column, with all other entries being \( 0 \). Thus this representation is in the essential image of \( \funct{free}^\mathrm{pt}_* \) by Proposition~\ref{prop:mult-basis}.
\end{proof}
\begin{example}
Let $F=\Z/5\Z$ and consider the following representation $R$ of $A_2$:
\[
\begin{tikzcd}
\Z/5\Z\ar[r,"\cdot 3"] & \Z/5\Z.
\end{tikzcd}
\]
It is easy to see that $R$ is additively $\Set$-realizable by a direct argument. However, following the proof strategy above, one observes that $R$ is a direct summand of the top row of the following diagram.
\[
\begin{tikzcd}[sep=25mm, ampersand replacement=\&]
(\Z/5\Z)^4\ar{d}[swap]{Z = \begin{bmatrix} 1 & 1 & 1 & 1 \\ 1 & 2 & 4 & 3 \\ 1 & 4 & 1 & 4 \\ 1 & 3 & 4 & 2
\end{bmatrix}} \ar{r}{\begin{bmatrix} 1 & 0 & 0 &0 \\ 0 & 3 & 0 & 0 \\ 0 & 0 & 4 & 0 \\ 0 & 0 & 0 & 2 \end{bmatrix}} \& (\Z/5\Z)^4 \ar{d}{\begin{bmatrix} 1 & 1 & 1 & 1 \\ 1 & 2 & 4 & 3 \\ 1 & 4 & 1 & 4 \\ 1 & 3 & 4 & 2
\end{bmatrix} =Z} \\
(\Z/5\Z)^4\ar{r}{\begin{bmatrix} 0 & 0 & 0 & 1 \\ 1 & 0 & 0 & 0 \\ 0 & 1 & 0 & 0 \\ 0 & 0 & 1 & 0 \end{bmatrix}} \& (\Z/5\Z)^4
\end{tikzcd}
\]
The change of basis given by $Z$ with $\zeta=2$ shows that $R$ is a summand of a representation with a coherent basis (lower row). 
\end{example}

\subsection{Sets with group actions}
\label{sec.groupac}

For this subsection, let \( G \) be a subgroup of the multiplicative group \( F^{\times} \) of our base field. Then any \( F \)-vector space is naturally a (pointed) \( G \)-set. (Here, by a pointed \( G \)-set we mean a \( G \) set together with a disinguished base point $*$, such that $g* = * \; \forall g \in G$.) Our forgetful functor decomposes into the two steps
\[ \VectF \xrightarrow{\funct{forget}^\mathrm{pt}_G} G\text -\pSet \xrightarrow{\funct{forget-action}} \pSet, \]
where the first functor forgets the vector space addition and restricts scalar multiplication to $G$, while the second functor forgets the $G$-action.

Both steps have left adjoints: The free pointed \( G \)-set on a pointed set \( S \) is \( G \times (S \setminus \{ * \}) \cup \{ * \} \). Another way of saying this is taking \( G \times S \), but identifying all elements of the form \( (g, * ) \). Ignoring this slight deviation for the base point we will denote this functor by \( G \times - \).

The free vector space on a pointed \( G \)-set \( S \) can be obtained from the free vector space on the pointed set \( S \) by factoring out the subspace generated by objects of the form \( fg \cdot s - f \cdot gs \) for \( f \in F \), \( g \in G \), and \( s \in S \). We denote this quotient space by \( \funct{free}^\mathrm{pt}_G(S) \).

By construction (and by the uniqueness of adjoints) we have that
\[ \funct{free}^\mathrm{pt} \cong \funct{free}^\mathrm{pt}_G \circ (G \times - ). \]

\begin{example}\label{ex.D4-2}
Let $R=(\Z/p\Z)^2$ as a vector space over $\Z/p\Z$. Then, as a $G$-set for $G=\{1,2, \ldots, p-1\}$, we have the following orbit in $R$ (seen as a $G$-set) for every $(i,j)\neq (0,0)$:
\[
\{(i,j), (2i,2j), (3i, 3j), \ldots, ((p-1)i, (p-1)j)\}.
\]
Under $\funct{free}^\mathrm{pt}_G$ the image of this orbit is a vector space of dimension 1; this is in contrast to the $(p-1)$-dimensional vector space one would get by disregarding the action of $G$.

More generally, if $F$ is a finite field and $G = F^{\times}$ is the multiplicative group, then
\( \funct{free}^\mathrm{pt}_G \funct{forget}^\mathrm{pt}_G F^d \) is a vector space whose basis is in bijection to the one-dimensional subspaces of our original space, that is, in bijection to $d-1$-dimensional projective space over the same field.
\end{example}

The previous example suggests that working with $G$-sets can improve the algorithm presented in Proposition~\ref{prop.sys_eq} when the underlying field is different from $\Z/2\Z$. We return to this in Section~\ref{sec:gset-algorithm}.

\begin{proposition} \label{prop.free-G_split}
Let \( G \) be a finite subgroup of \( F^{\times} \). Then the natural epimorphism of functors from pointed \( G \)-sets to vector spaces
\[ \funct{free}^\mathrm{pt} \circ \funct{forget-action} \to \funct{free}^\mathrm{pt}_G \]
naturally splits. (Here \( \funct{forget-action} \) denotes the forgetful functor from pointed \( G \)-sets to pointed sets.)
\end{proposition}

\begin{proof}
Note that since \( G \) is a finite subgroup of \( F^{\times} \) the characteristic of \( F \) cannot divide the order of \( G \). In particular we can write \( \frac{1}{|G|} \) in \( F \). 

By our construction above, \( \funct{free}^\mathrm{pt}_G \) consists of formal linear combinations of the elements of \( S \setminus \{*\} \), modulo the subspace \( U \) generated by elements of the form \( fg \cdot s - f \cdot gs \). 
We consider the ``averaging map'' given on generators of \( \funct{free}^\mathrm{pt} \circ \funct{forget-action} (S)) \) as
\begin{align*}
\funct{free}^\mathrm{pt} \circ \funct{forget-action} (S) & \to \funct{free}^\mathrm{pt} \circ \funct{forget-action}(S) \\
1 \cdot s & \mapsto \frac{1}{|G|} \sum_{g \in G} g \cdot g^{-1} s
\end{align*}
Note that this map vanishes on the generators of \( U \):
\begin{align*}
fg \cdot s - f \cdot gs \mapsto & fg \frac{1}{|G|} \sum_{h \in G} h \cdot h^{-1} s - f \frac{1}{|G|} \sum_{h \in G} h \cdot h^{-1} g s \\
& = \frac{1}{|G|} f \left( \sum_{h \in G} gh \cdot h^{-1}s - \sum_{\tilde{h} \in G} g \tilde{h} \cdot \tilde{h}^{-1} s \right) \\
& = 0,
\end{align*}
where for the first equality we choose \( \tilde{h} = g^{-1} h \), and observe that when \( h \) runs over all of \( G \) then so does \( \tilde{h} \).

Thus we have a well-defined map
\[ \funct{free}^\mathrm{pt}_G(S) = \frac{ \funct{free}^\mathrm{pt} \circ \funct{forget-action}(S)}{U} \to  \funct{free}^\mathrm{pt} \circ \funct{forget-action}(S). \]
It is easily checked that this map is natural in \( S \), and right inverse to the natural epimorphism \( \funct{free}^\mathrm{pt} \circ \funct{forget-action}(S) \to \funct{free}^\mathrm{pt}_G(S) \). 
\end{proof}

\begin{corollary} \label{cor.image_from_G_is_same}
Let \( X \) be a small category. Then for any finite subgroup \( G \) of \( F^{\times} \), the additive images of the functors \( \funct{free}_* \) and \( (\funct{free}^\mathrm{pt}_G)_* \) in \( \rep(X, \VectF) \) coincide.

\end{corollary}

\begin{proof}
First recall that by Proposition~\ref{prop.same_add_image} the additive images of \( \funct{free}_* \) and \( \funct{free}^\mathrm{pt}_* \) coincide. 
By the equality \( \funct{free}^\mathrm{pt} = \funct{free}^\mathrm{pt}_G \circ (G \times -) \) we know that the image of \( \funct{free}^\mathrm{pt} \) is contained in that of \( \funct{free}^\mathrm{pt}_G \), and hence the same holds for additive images.

Now let \( R \in \rep(X, G\text{-}\pSet) \). By Proposition~\ref{prop.free-G_split} we know that the epimorphism
\[ \funct{free}^\mathrm{pt} \circ \funct{forget-action}(R) \to \funct{free}^\mathrm{pt}_G(R) \]
splits, so \( \funct{free}^\mathrm{pt}_G(R) \) lies in the additive image of \( \funct{free}^\mathrm{pt} \). It follows that any summand of  \( \funct{free}^\mathrm{pt}_G(R) \) lies in the same additive image.
\end{proof}

\begin{corollary}
Let \( R \in \rep(X, \VectF) \). Assume \( R \) has an almost coherent basis.
If all the non-zero scalars appearing in the matrices of the structure maps of $R$ for that basis generate a finite subgroup of \( F^{\times} \), then \( R \) is additively \( \Set \)-realizable.
\end{corollary}

\begin{proof}
Let \( G \) be the group generated by the non-zero scalars in the matrices. 
The conditions for an almost coherent basis describe exactly what it means to be in the image of \( (\funct{free}^\mathrm{pt}_G)_* \). Thus this corollary follows from the previous one.
\end{proof}

As an immediate consequence, we again obtain Corollaries~\ref{cor.near_mult_basis_is_enough} and~\ref{cor.finite_cat_is_enough}.

\begin{remark}
Throughout this subsection we considered pointed \( G \)-sets. Very similar results also hold for (non-pointed) \( G \)-sets.
\end{remark}

\subsection{An improved algorithm for testing membership in the additive image}
\label{sec:gset-algorithm}

As a consequence of Corollary~\ref{cor.image_from_G_is_same} we obtain a (computationally cheaper, unless \( F = \Z/2\Z \), but more technical to implement) test for being additively \( \Set \)-realizable:

\begin{corollary}
Let \( G \) be a finite subgroup of \( F^{\times} \). A representation \( R \in \rep(X, \VectF ) \) lies in the additive image of \( \funct{free}_* \) if and only if the natural epimorphism
\[ \funct{free}^\mathrm{pt}_G \circ \funct{forget}^\mathrm{pt}_G \circ R \to R \]
splits.
\end{corollary}
\begin{proof}
This follows from Corollary~\ref{cor.adj_summand} and Corollary~\ref{cor.image_from_G_is_same}.
\end{proof}

Let us analyze the running time for finite fields: As in Section~\ref{subsect:runtime} we have a map vector space representations, and need to decide if it splits. 
Choosing $G =F^{\times}$,
we follow the exact same calculations as in that section, the only change being that the dimension of  
\( \funct{free}^\mathrm{pt}_G \circ \funct{forget}^\mathrm{pt}_G (F^{d_v}) \)
is \( \frac{|F|^{d_v} - 1}{|F|-1} \), rather than \( |F|^{d_v} - 1 \) as in \ref{subsect:runtime}. Thus we get a system of linear equations with
\begin{align*}
\#{\rm indeterminants} &= \sum_{v \in \operatorname{Obj}(X)} d_v \left(\frac{|F|^{d_v} - 1}{|F|-1}\right)  \\
\#{\rm equations} &= \sum_{v \in \operatorname{Obj}(X)} d_v^2 + \sum_{v \to w } d_v \left(\frac{|F|^{d_w} - 1}{|F|-1}\right).
\end{align*}

Thus we get the following slight improvement on Corollary~\ref{cor.algo}.

\begin{proposition}
\label{prop:improvedalgo}
Using Gaussian elimination, one can determine whether $R\in \rep(X,\vectF)$ is additively $\Set$-realizable in $\mathcal{O}\left(m^3\cdot |X|^3\cdot |F|^{3(m-1)} \right)$ arithmetic operations, where $m = \max_{v \in \operatorname{Obj}(X)} \dim R(v)$.
\end{proposition}

Let us illustrate the reduced sizes in the improved algorithm taking a second look at a previous example:

\begin{example} \label{example:D4_by_G-set}
Revisiting $R$ from Example~\ref{ex.D4} with $\Z/2\Z$ replaced by any finite field $F$,
%
we consider the representation
\[ \begin{tikzcd}[sep=10mm]
& F\ar{d}{\begin{bmatrix} 1 \\ 0 \end{bmatrix}} & \\
F \ar{r}{\begin{bmatrix} 0 \\ 1 \end{bmatrix}} & F \oplus F & F \, . \ar{l}[swap]{\begin{bmatrix} 1 \\ 1 \end{bmatrix}} 
\end{tikzcd} \]
Choosing $G = F^\times$,
by the discussion in Example~\ref{ex.D4-2} we know that the vector space \( \funct{free}^\mathrm{pt}_G \circ\funct{forget}^\mathrm{pt}_G (F) \) is one-dimensional, and \( \funct{free}^\mathrm{pt}_G \circ \funct{forget}^\mathrm{pt}_G (F^2) \) has a basis corresponding to the lines in $F^2$. In particular, since the three vectors \( \left[ \begin{smallmatrix} 1 \\ 0 \end{smallmatrix} \right] \), \( \left[ \begin{smallmatrix} 0 \\ 1 \end{smallmatrix} \right] \), and \( \left[ \begin{smallmatrix} 1 \\ 1 \end{smallmatrix} \right] \) are linearly independent, they are mapped to different coordiantes of \( \funct{free}^\mathrm{pt}_G \circ \funct{forget}^\mathrm{pt}_G (F^2) \). It follows that the representation decomposes into a sum of indicator modules. Therefore, $R$ is not additively $\Set$-realizable.
\end{example}

\section{Transferring \texorpdfstring{\( \Set \)}{Set}-realizability}
\label{sec:transfer}

The aim of this section is to investigate how \( \Set \)-realizability changes when we make (small) changes to the base category \( X \). In the first subsection we will consider general full subcategories. In the second subsection we will look at two categories differing by a terminal object --- this setup will in particular lead us to a surprisingly large class of categories admitting only finitely many indecomposable \( \Set \)-realizables. In the final section we will consider reflection functors, changing the direction of certain arrows in our base category.

\subsection{Full subcategories}
\label{sec:fullsubcat}
In this section, $X$ denotes a small category and $S$ a full subcategory of $X$. Recall that if $\mathbf{C}$ is a cocomplete category, then the restriction functor $\funct{res}\colon \rep(X, \mathbf{C}) \to \rep(S, \mathbf{C})$ has a left adjoint given by the left Kan extension $\funct{Lan}\colon\rep(S, \mathbf{C})\to \rep(X, \mathbf{C})$. The unit of the adjunction is an isomorphism, and, in particular, $\funct{Lan}$ is fully faithful. 

\begin{lemma}\label{lem:kan-commutes}
Both the restriction functor $\funct{res}\colon \rep(X, \Set) \to \rep(S, \Set) \) and its left adjoint $\funct{Lan}\colon \rep(S, \Set)\to \rep(X, \Set)$ commute with the functor $\funct{free}_*$.
\end{lemma}

\begin{proof}
It is clear that $\funct{res}\circ \funct{free}_* = \funct{free}_*\circ \funct{res}$. 

For the left Kan extensions, recall that the left adjoint of a composition of functors is the composition of left adjoints. Thus the equality
\[ \funct{forget}_* \circ \funct{res} = \funct{res} \circ \funct{forget}_* \]
of functors from $\rep(X, \VectF)$ to $\rep(S, \Set)$ implies that the two compositions of left adjoints are also equal, that is that
\[ \funct{Lan} \circ \funct{free}_* = \funct{free}_* \circ \funct{Lan} \]
as functors from \( \rep(S, \Set) \) to \( \rep(X, \VectF) \). 
\end{proof}

\begin{proposition} \label{prop.Lan_is_realizable}
A linear representation $R \colon S\to \VectF$ is additively \( \Set \)-realizable if and only if \( \funct{Lan} R \colon X\to \VectF \) is. 
\end{proposition}

\begin{proof}
If \( R \) is a direct summand of \( \funct{free} \circ T \), then \( \funct{Lan} (R) \) is a direct summand of $\funct{Lan}(\funct{free} \circ T)$ as left adjoints preserve coproducts.  The forward implication now follows from Lemma~\ref{lem:kan-commutes}. Conversely, if \( \funct{Lan} (R) \) is a direct summand of \( \funct{free} \circ T \), then \( R = \funct{res} \circ \funct{Lan} (R) \) is a direct summand of $\funct{res} (\funct{free} \circ T)$. The backward implication now follows from Lemma~\ref{lem:kan-commutes}.
\end{proof}

The previous proposition implies the following.

\begin{corollary} \label{cor.subcategory}
Let \( S \) be a full subcategory of \( X \).
\begin{itemize}
\item If all indecomposable additively \( \Set \)-realizable representations of \( X \) are indicator representations, then the same is true for \( S \).
\item If there are only finitely many indecomposable additively \( \Set \)-realizable representations of \( X \), then the same is true for \( S \).
\item If all linear representations of \( X \) are additively \( \Set \)-realizable, then the same is true for \( S \).
\end{itemize}
\end{corollary}

\begin{example}
For the quiver of type \( \tilde{D}_4 \) with three ingoing arrows as depicted below, over \( F = \Z / 2\Z \), there is a linear representation that is not additively \( \Set \)-realizable.
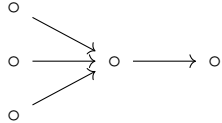
\[ \begin{tikzcd}[row sep=3mm]
\circ \ar[rd] \\
\circ \ar[r] & \circ \ar[r] & \circ \\
\circ \ar[ru]
\end{tikzcd} \]
This follows from Example~\ref{ex.D4} by the corollary. (For an arbitrary field the same conclusion holds once we have Example~\ref{ex.D4_inward}.)
\end{example}

\subsection{Adding and removing terminal objects in finite categories}
\label{sec:terminalobj}

We now focus specifically on terminal objects in our base category \( X \). In this subsection, we assume \( X \) to be finite and we consider finite-dimensional representations, so that we get unique decompositions into indecomposables.

\begin{lemma} \label{lem.terminal_is_1d}
Assume \( X \) has a terminal object \( t \). Let \( R \) be an indecomposable linear representation of \( X \). If \( R \) is additively \( \Set \)-realizable then \( R(t) \) is at most one-dimensional.
\end{lemma}

\begin{proof}
Let \( N\in\rep(X, \set)\). Then we obtain a disjoint union decomposition of \( N \) as
\[ N = \bigcup_{h \in N(t)} N_h, \]
where $N_h$ is the \(\set\)-representation of $X$ given by \( N_h(v) = N(v \to t)^{-1}(\{h\}) \). It follows that
\[ \funct{free} \circ N = \bigoplus_{h \in N(t)} \funct{free} \circ N_h. \]
Now, if \( R\) is a direct summand of \( \funct{free} \circ N \) then it is a direct summand of one of the \( \funct{free} \circ N_h \), hence it is at most one-dimensional in position \( t \).
\end{proof}

\begin{example} \label{ex.D4_inward}
Consider the quiver of type \( D_4 \) as depicted below, and the indecomposable representation with the dimension vector given next to it.
\[ \begin{tikzcd}[sep=5mm]
& \circ \ar[d] &&&& 1 \ar[d,-] & \\
\circ \ar[r] & \circ & \circ \ar[l] && 1 \ar[r,-] & 2 & 1\ar[l,-]
\end{tikzcd} \]
This linear representation is not additively \( \Set \)-realizable. This gives a deeper explanation for the conclusion drawn for $F=\Z/2\Z$ in Example~\ref{ex.D4}.

It further extends to arbitrary quivers of type \( D \), see Example~\ref{example:D_n_not_realizable}.
\end{example}

\begin{example}
Consider the quivers of types \( E_6 \), \( E_7 \), and \( E_8 \) as depicted on the left below, and the indecomposable representation with the dimension vector given next to it.
\[ \begin{tikzcd}[sep=3mm]
&& \circ &&&&&& 2 && \\
\circ \ar[r] & \circ \ar[r] & \circ \ar[u] & \circ \ar[l] & \circ \ar[l] && 1 \ar[r,-] & 2 \ar[r,-] & 3 \ar[u,-] & 2 \ar[l,-] & 1 \ar[l,-]
\end{tikzcd} \]
\[ \begin{tikzcd}[sep=3mm]
&&& \circ \ar[d] &&&&&&& 2 \ar[d,-] && \\
\circ \ar[r] & \circ \ar[r] & \circ \ar[r] & \circ \ar[r] & \circ \ar[r] & \circ && 1 \ar[r,-] & 2 \ar[r,-] & 3 \ar[r,-] & 4 \ar[r,-] & 3 \ar[r,-] & 2
\end{tikzcd} \]
\[ \begin{tikzcd}[sep=3mm]
&&&& \circ \ar[d] &&&&&&&& 3 \ar[d,-] && \\
\circ & \circ \ar[l] & \circ \ar[l] & \circ \ar[l] & \circ \ar[l] & \circ \ar[l] & \circ \ar[l]  && 2 & 3 \ar[l,-] & 4 \ar[l,-] & 5 \ar[l,-] & 6 \ar[l,-] & 4 \ar[l,-] & 2 \ar[l,-]
\end{tikzcd} \]
These linear representations are not additively \( \Set \)-realizable.
\end{example}

More generally we have the following.

\begin{proposition} \label{prop.colim_1d}
Let \( R \) be an indecomposable linear representation of \( X \). If \( X \) is additively \( \Set\)-realizable, then \( \varinjlim_{v \in X} R(v) \) is at most one-dimensional.
\end{proposition}

\begin{proof}
We may extend \( X \) to a category \( X \cup \{ t \} \) with a terminal object \( t \). By Proposition~\ref{prop.Lan_is_realizable} we have that \( R \) being additively realizable implies that also \( \funct{Lan} F \) is additively \( \Set \)-realizable, where we take the left Kan extension along the inclusion of \( X \) into \( X \cup \{ t \} \).

Moreover, since $\funct{res}\circ \funct{Lan} \cong \mathrm{id}_{\rep(X,\vectF)}$ and and $\funct{res}$ preserves direct sums, we know that \( \funct{Lan} F \) is indecomposable. It follows from Lemma~\ref{lem.terminal_is_1d} that \( (\funct{Lan} F)(t) =\varinjlim_{v \in X} R(v) \) is at most one-dimensional.
\end{proof}

\begin{example}
Consider the quiver of type \( E_8 \) as depicted below, and the indecomposable representation with the dimension vector given next to it.
\[ \begin{tikzcd}[sep=3mm]
&&&& \circ &&&&&&&& 3 && \\
\circ \ar[r] & \circ \ar[r] & \circ \ar[r] & \circ \ar[r] & \circ \ar[r] \ar[u] & \circ & \circ \ar[l]  && 1 \ar[r,-] & 2 \ar[r,-] & 3 \ar[r,-] & 4 \ar[r,-] & 5 \ar[r,-] \ar[u,-] & 4 & 2 \ar[l,-]
\end{tikzcd} \]
This linear representation has a two-dimensional colimit, and therefore is not additively \( \Set \)-realizable.
\end{example}

\begin{theorem} \label{thm.add_terminal}
Assume \( X \) has a terminal object \( t \), and let \( S = X \setminus \{ t \} \). If there are only finitely many indecomposable additively \( \Set \)-realizable representations of \( S \) then the same is true for \( X \).
\end{theorem}

\begin{proof}
Assume \( R \) is an indecomposable additively \( \Set \)-realizable representation of \( X \). By Lemma~\ref{lem.terminal_is_1d} we know that \( R(t) \) is zero or one-dimensional. If \( R(t) \) is zero, then \( R \) is really a representation of \( S \), and there are only finitely many indecomposable additively \( \Set \)-realizable representations of \( S \).

The only other possibility to consider is \( R(t) = F \). We will show that there are also only finitely many possibilities in this case, by establishing the following two statements.

\paragraph{Claim 1}
\( R \) is uniquely determined by its restriction \( \funct{res}(R) \) to \( S \).

\paragraph{Claim 2}
The restriction \( \funct{res}(R) \) is multiplicity free, that is there are no two isomorphic summands in a decomposition of \( \funct{res}(R) \) into indecomposables.

Once these two claims are established we can easily conclude that there are only finitely many possibilities for indecomposable \( R \), since there are only finitely many multiplicity free representations of \( S \).

For both claims, note that \( R \) is given by \( \funct{res}(R) \), together with information on all the structure maps from other objects to \( R(t) = F \). By definition of the colimit, giving this collection of structure maps is equivalent to giving a single linear map
\[ \varinjlim_{s \in S} R(s) \to R(t). \]
Note that \( \varinjlim \) is an additive functor, so it can be computed summand-wise. Moreover, by Proposition~\ref{prop.colim_1d}, for any indecomposable direct summand \( R' \) of \( \funct{res}(R) \) we have \( \varinjlim_{s \in S} R'(s) \in \{ 0, F \} \). If \( \varinjlim_{s \in S} R'(s) = 0 \) then this summand is not connected to \( R(t) \) at all, contradicting the indecomposability of \( R \). If \( \varinjlim_{s \in S} R'(s) = F \) then there is a non-zero linear map \( \varinjlim_{s \in S} R'(s) \to R(t) \), and this linear map is unique up to multiplying \( R' \) by a non-zero scalar. This establishes Claim 1 above.

For Claim 2, assume contrarily that two copies of an indecomposable \( R' \) appear in a decomposition of \( \funct{res}(R) \), with structure maps given by two scalar multiplications \( \alpha_i \colon \varinjlim_{s \in S} R'(s) \to R(t) = F \). Then there is a diagonal direct summand \( \left( \begin{smallmatrix} \alpha_2 \\ - \alpha_1 \end{smallmatrix} \right) \colon R' \to {R'}^2 \) on which the structure map vanishes, that is a direct summand that splits off. This contradicts the indecomposability of \( R \). Thus there cannot be two isomorphic summands in a decomposition of \( \funct{res}(R) \), establishing Claim 2.
\end{proof}

A \emph{tree quiver} is a quiver $Q$ for which the underlying undirected graph is a tree in the graph theoretic sense. 
\begin{corollary}
\label{cor.treequiver}
Let \( Q \) be a tree quiver with a single sink $s$. Then there are only finitely many additively \( \Set \)-realizable representations of \( Q \).
\end{corollary}

\begin{proof}
This follows from Theorem~\ref{thm.add_terminal} by induction. Note that \( Q \setminus \{ s \} \) is a disjoint union of smaller tree quivers each having a unique sink.
\end{proof}

A version of this corollary giving an explicit classification of indecomposables has been obtained independently in \cite{bindua2024decomposing}.

\subsection{Restrictions and edge contractions}
\label{sec.edge_contractions} In this section, we drop the restriction that $X$ is finite.

The following diagram shows the Hasse diagram of a poset \(X\) and its relationship to \(D_4\) via edge contractions. Specifically, we introduce an intermediate step by defining a pre-ordered set \(\widehat{X}\), in which the two contracted arrows are made invertible.  

\[ \begin{tikzcd}[sep=3mm]
X&& \circ\ar[d] &&&&\widehat{X}&& \circ \ar[d] && && D_4 & \circ\ar[d] &  \\
\circ \ar[r] & \circ & \circ \ar[l] \ar[r] & \circ & \circ \ar[l] && \circ \ar[r] & \circ\ar[r] & \circ \ar[l, bend left]\ar[r, bend right] & \circ \ar[l] & \circ \ar[l] & & \circ \ar[r] & \circ & \ar[l] \circ
\end{tikzcd} \]

We have seen that in \(D_4\), with this specific arrow orientation, there exists an indecomposable that is not additively \(\Set\)-realizable. Does this imply that \(X\) also contains an indecomposable outside the additive image of \(\Set\)? In this section, we show that this is indeed the case.

First, we consider the general setting of a functor \( \funct{F} \colon X \to Y \) of small categories. For any other category \( \mathbf{T} \) we have the induced restriction functor \( \rep(Y, \mathbf{T}) \to \rep(X, \mathbf{T}) \) obtained by pre-composition with $\funct{F}$. 

\begin{lemma} \label{lemma:restriction_preserves_realizable}
Let \( \funct{F} \colon X \to Y \) be a functor. Then the induced functor \( \rep(Y, \VectF) \to \rep(X, \VectF) \) preserves additive \( \Set \)-realizability. 
\end{lemma}

\begin{proof}
Assume \( R \in \rep(Y, \VectF) \) is additively realizable, i.e. $R\oplus R' \cong \funct{free} \circ S$ for some $S\in \rep(Y, \Set)$. Then, 
since direct sums are taken pointwise,
\[(R\circ \funct{F})\oplus (R'\circ \funct{F}) = (R\oplus R')\circ \funct{F} \cong \funct{free} \circ S \circ\funct{F}.\]
It follows that $R\circ \funct{F}$ is additively $\Set$-realizable.
\end{proof}

\begin{proposition} \label{prop:restriction_preserves_realizability}
Let \( \funct{F} \colon X \to Y \) be a functor such that the induced functor \( \rep(Y, \VectF) \to \rep(X, \VectF) \) is fully faithful. Then the induced functor both \emph{preserves} and \emph{reflects} additive \( \Set \)-realizability. That is, a representation \( R \in \rep(Y, \VectF) \) is additively \( \Set \)-realizable if and only if \( R \circ \funct{F} \in \rep(X, \VectF) \) is.
\end{proposition}

\begin{proof}
The ``preserves'' is just a restatement of Lemma~\ref{lemma:restriction_preserves_realizable}. To prove the other direction, we recall from Corollary~\ref{cor.criterion_for_add_image} that $R\circ \funct{F}$ is additively $\Set$-realizable if and only if precomposing the counit $\varepsilon$ of the adjunction $\funct{free} \dashv \funct{forget}$ with $R\circ \funct{F}$, 
\[ \varepsilon \circ R\circ \funct{F} \colon \funct{free} \circ \funct{forget} \circ R \circ \funct{F} \to R\circ \funct{F},\]
yields a split epimorphism. Let 
\[\phi \colon  R\circ \funct{F} \to  \funct{free}\circ \funct{forget}  \circ R\circ \funct{F}\]
be such a splitting. 
We get from the fullness of pre-composition with $\funct{F}$ that  $\phi = \psi \circ \funct{F}$ for some $\psi \colon R \to  \funct{free}\circ \funct{forget} \circ R$. 
Now consider the horizontal composition of natural transformations
$(\varepsilon \circ R) \circ \psi \colon R \to R$.
Precomposing with $\funct{F}$ yields
%
\[((\varepsilon \circ R) \circ \psi)\circ \funct{F} = (\varepsilon \circ R \circ \funct{F}) \circ (\psi\circ \funct{F}) = (\varepsilon \circ R \circ \funct{F}) \circ \phi = {\rm id}_{R\circ \funct{F}}.\]
Since precomposition with $\funct{F}$ was assumed to be faithful, it follows that 
\[(\varepsilon \circ R) \circ \psi = {\rm id}_{R}. \qedhere \]
\end{proof}

We summarize two immediate implications of the previous proposition in the following corollary. 

\begin{corollary} \label{cor:transfer_to_localization}
Let \( \funct{F} \colon X \to Y \) be a functor such that the induced functor \[ \rep(Y, \VectF) \to \rep(X, \VectF) \qquad\qquad R\mapsto R\circ \funct{F}\] is fully faithful.

\begin{itemize}
\item If there are only finitely many indecomposable additively \( \Set \)-realizables in \( \rep(X, \vectF) \) then the same is true for \( \rep(Y, \vectF) \).
\item If all indecomposables in \( \rep(X, \vectF) \) are additively \( \Set \)-realizable then the same is true for \( \rep(Y, \vectF) \).
\end{itemize}
\end{corollary}

The following proposition provides a (general type of) situation in which the assumptions of Proposition~\ref{prop:restriction_preserves_realizability} and Corollary~\ref{cor:transfer_to_localization} are met.

\begin{proposition} \label{prop:induces_fully_faithful_restriction}
Let \( \funct{F} \colon X \to Y \) be a functor, which is bijective on objects, and such that every morphism in $Y$ can be written as a finite composition $\beta_1\circ \cdots \circ \beta_k$ where each $\beta_i$ equals $\funct{F}(\alpha)$ or $\funct{F}(\alpha)^{-1}$ for some morphism $\alpha$ in $X$.  Then, the functor \( \rep(Y, \VectF) \to \rep(X, \VectF) \) is fully faithful.
\end{proposition}

\begin{proof}
For two representations \( R \) and \( S \) we have
\[ \Hom_{\rep(Y, \VectF)}(R, S) = \{ (\varphi_v)_{v \in \operatorname{Obj}(Y)} \mid \forall v \overset{y}{\to} w \in Y \colon \varphi_w \circ R(y) = S(y) \circ \varphi_v \}. \]
Clearly it suffices to check the commutativity condition on a collection of morphisms from which every other morphism can be expressed as a finite composition.  Moreover, if the commutativity condition holds for an isomorphism, then it also holds for the inverse of this isomorphism. These two statements, together with the assumption of the proposition, mean that we may rewrite the right-hand side above as
\[\{ (\varphi_v)_{v \in \operatorname{Obj}(Y)} \mid \forall v \overset{x}{\to} w \in X \colon \varphi_{\funct{F}(w)} \circ R(\funct{F}(x)) = S(\funct{F}(x)) \circ \varphi_{\funct{F}(v)} \}. \]
This is precisely the condition describing \( \Hom_{\rep(X, \VectF)}(R \circ \funct{F}, S \circ \funct{F}) \).
\end{proof}

Proposition~\ref{prop:induces_fully_faithful_restriction} covers two natural settings:
\begin{itemize}
\item Localizations: certain morphisms of the original category \( X \) are made invertible, and
\item Quotients: certain morphisms of the original category \( X \) are identified with each other. 
\end{itemize}

Now let us consider the particularly nice situation of quivers and posets and our initial question concerning edge contractions.

\begin{definition} \label{def:edge_contraction_quiver}
Let \( X \) be a quiver, and \( \alpha \) an arrow between two distinct objects in \( X \) (i.e., not a self-loop). The \emph{edge contraction} of \( X \) at \( \alpha \) is given by the quiver $Y$ obtained from \( X \) by removing the arrow \( \alpha \) and identifying its two end-points.
\end{definition}

\begin{definition}
Let \( X \) be a poset, and \( v \lneq w \) such that there are no further vertices between \( v \) and \( w \). The \emph{edge contraction} of \( X \) at \( v \lneq w \) is the poset $Y$ obtained from \( X \) by
\begin{itemize}
\item identifying \( v \) and \( w \), and
\item declaring \( x \leq y \) whenever \( x \leq w \) and \( v \leq y \) (in addition to the pre-existing relations from \( X \)).
\end{itemize}
\end{definition}

Observe that by treating $X$ as a category in either of the two cases, we have a functor $\funct{c}\colon X\to Y$ induced by mapping the two endpoints $v$ and $w$ of the contracted edge to the same point in $Y$. The pre-composition functor $R\mapsto R\circ \funct{c}$ thus assigns to $v$ and $w$ the same vector space, and the identity map to the contracted edge.

\begin{theorem}
\label{thm.edgecontraction}
Let \( X \) be a quiver or a poset, and \( \funct{c}\colon X \to Y \) an edge contraction. Then the restriction functor \[ \rep(Y, \VectF) \to \rep(X, \VectF) \qquad \qquad R\mapsto R\circ \funct{c}\] preserves and reflects additive \( \Set \)-realizability. In particular, 
\begin{itemize}
\item If there are only finitely many indecomposable additively \( \Set \)-realizables in \( \rep(X, \vectF) \) then the same is true for \( \rep(Y, \vectF) \).
\item If all indecomposables in \( \rep(X, \vectF) \) are additively \( \Set \)-realizable then the same is true for \( \rep(Y, \vectF) \).
\end{itemize}
\end{theorem}

\begin{proof}
We prove the claim for \( X \) being a poset; the case for quivers follows analogously by adjusting the notation accordingly.

Although our setting is not directly covered by the propositions above, the following variation applies: Let \( \widehat{X} \) be the minimal preorder containing all the relations in \( X \) in addition to the relation \( w \leq v \). Then the canonical functor \( X \to \widehat{X} \) satisfies the hypotheses of Proposition~\ref{prop:induces_fully_faithful_restriction}, which implies that the conditions of Proposition~\ref{prop:restriction_preserves_realizability} hold as well. 

Finally, the functor $\funct{c}\colon X\to Y$ factors through a functor $\hat{\funct{c}}\colon \widehat{X}\to Y$ by mapping $w\leq v$ to the same identity morphism as $v\leq w$. We observe that this functor \( \hat{\funct{c}} \) is an equivalence. In particular, it induces an equivalence of categories $\rep(Y, \VectF) \to \rep(\widehat{X}, \VectF)$. By Proposition~\ref{prop:restriction_preserves_realizability}, it follows that the first part, and therefore both parts, of the following composition preserve and reflect additive $\Set$-realizability:
\[\rep(Y, \VectF) \to \rep(\widehat{X}, \VectF) \to \rep(X, \VectF). \qedhere \]
\end{proof}

\begin{example} \label{example:D_n_not_realizable}
Let \( X \) be a quiver of type \( D_n \), such that all arms contain at least one arrow oriented inwards. Equivalently, \( X \) has a quiver of type \( D_4 \) oriented as in Example~\ref{ex.D4_inward} as an iterated edge contraction (contracting all arrows except the one ingoing arrow for each arm). Then there is an indecomposable in \( \rep(X, \vectF) \) which is not additively \( \Set \)-realizable.
\end{example}

\begin{remark}
Comparing Theorem~\ref{thm.edgecontraction} to Question~\ref{questions} one may wonder if the corollary should also contain a statement about when all indecomposable additively \( \Set \)-realizable representations are indicator representations. However, the statement one might suspect is not correct: For \( X \) the poset category \( \{ 1, 2, 3 \} \) one easily sees that all indecomposables are indicator representations. But if we localize the morphism \( 1 \leq 3 \), the resulting category can be depicted as
\[ \begin{tikzcd}
1 \ar[r,"\alpha"] & 2 \ar[r,"\beta"] & 3 \ar[ll,bend left,"(\beta \alpha)^{-1}"]
\end{tikzcd} \]
which we will see in Section~\ref{subsect:only_thin} to have non-indicator representations even among its indecomposables which are additively \( \Set \)-realizable. 
\end{remark}

\subsection{Reflection functors}
\label{sect.APR}
In this section, we will show that certain reflection functors (also known as \emph{Auslander--Platzeck--Reiten tilts}) preserve (additive) \( \Set \)-realizability.

In this section we assume our category to be the path category of a quiver, and hence denote it by \( Q \) rather than \( X \). 

Let \( d \) be a source of a quiver \( Q \). We denote by \( \mu_d( Q) \) the quiver obtained from \( Q \) by reversing all arrows starting in \( d \). For \( R \in \rep( Q, \vectF ) \) we denote by \( \mu_d(R) \) the representation of \( \mu_d(Q) \) given on vertex \( v \) by
\[ \mu_d(R)(v) = \begin{cases} R(v) & \text{ for } v \neq d \\ \operatorname{Cok}[ R(d) \to \bigoplus_{\alpha \colon d \to ?} R({\operatorname{target}}(\alpha))] & \text{ for } v = d \end{cases} \]
Here the map of which we take the cokernel is given, componentwise, by the linear maps corresponding to all the arrows out of vertex \( d \). Clearly, for any arrow \( \alpha \) starting in \( d \), we have a natural linear map from \( R(\operatorname{target}(\alpha)) \) to this cokernel --- this natural linear map will be the structure map of the reversed arrow in \( \mu_d(R) \).
\begin{theorem}[Auslander--Platzeck--Reiten]
\label{thm.APR}
The construction \( \mu_d \) as described above induces an equivalence
\begin{align*}
& \{ R \in \rep( Q, \vectF) \mid S_d \text{ is not a direct summand of \( R \)} \}  \\
& \longrightarrow \{ R \in \rep( \mu_d(Q), \vectF) \mid S_d \text{ is not a direct summand of \( R \)} \},
\end{align*}
where \( S_d \) denotes the simple representation concentrated at vertex \( d \) for the respective quivers.
\end{theorem}
First we observe that this equivalence need not respect additive \( \Set \)-realizability. 
\begin{example}
\[\begin{tikzcd}[column sep=12mm,row sep=12mm,ampersand replacement=\&] \& F \& \\ F \& F\ar[u, "1"]\ar[l,swap,"1"]\ar[r,"1"] \& F\end{tikzcd}\xmapsto{\qquad\mu_d\qquad} \begin{tikzcd}[column sep=12mm,row sep=12mm,ampersand replacement=\&] \& F\ar[d,"{\begin{bmatrix} 1 \\ 1\end{bmatrix}}"] \& \\F\ar[r,"{\begin{bmatrix} 1 \\ 0\end{bmatrix}}"]  \& F^2\& F\ar[l,swap,"{\begin{bmatrix} 0 \\ 1\end{bmatrix}}"]\end{tikzcd} \]
The left representation is trivially in the essential image of $\funct{free}_*$ while we know from Example~\ref{ex.D4_inward} that the right representation is not additively $\Set$-realizable.
\end{example}

(Additive) \( \Set \)-realizability is preserved if we limit the degree of the source to at most 2. 

\begin{theorem} \label{thm.APR_preserves_induced}
Let \( Q \) be a quiver, and \( d \) a source of \( Q \) with at most two arrows starting in it. If $R\in \rep(Q, \vectF)$ is in the (additive) image of $\funct{free}_*$, then so is $\mu_d(R)$. In particular, if the functor \( \funct{free}_* \colon \rep( Q, \pset) \to \rep( Q, \vectF ) \) is essentially (additively) surjective, then so is the functor \( \funct{free}_* \colon \rep( \mu_d(Q), \pset) \to \rep( \mu_d(Q), \vectF ) \).
\end{theorem}

\begin{proof}
The proof comes down to observing that in the case of one or two arrows starting in a source $d$, the cokernel construction in $\mu_d$ is actually a colimit. The result then follows from preservation of colimits. 


For the case of one arrow $\alpha$ starting in \( d \), note that \( \pset \) has cokernels (given explicitly by identifying all elements of the image of a map with the base point). Thus, for a representation \( S \colon Q \to \pset \) we can form \( \mu_d(S) \) on objects as
\[ \mu_d(S)(v) = \begin{cases} S(v) & \text{for } v \neq d \\ \operatorname{Cok} S(\alpha) & \text{for } v = d \end{cases}. \]
On arrows, $\mu_d(S)(\alpha^{-1})$ is canonical projection  
\[S(\operatorname{target}(\alpha))\to \operatorname{Cok} S(\alpha),\] while the other maps are inherited from $S$.

For the case of two arrows $\alpha$ and $\beta$ starting in \( d \), note that the cokernel in the definition of \( \mu_d \) at the level of vector spaces is a push-out. Importantly, the category $\pset$ also admits push-outs (and more generally all small colimits), constructed as in $\set$ but with the additional identification of basepoints. Hence, for a representation \( S \colon Q \to \pset \) we can form \( \mu_d(S) \) as
\[ \mu_d(S)(v) = \begin{cases} S(v) & \text{for } v \neq d \\ \text{pushout of } S(\alpha) \text{ and }  S(\beta) & \text{for $v = d$ } \end{cases} \]
where the maps $\mu_d(S)(\alpha^{-1})$ and $\mu_d(S)(\beta^{-1})$ are given by canonical projections and the other maps are inherited from $S$. 


Since \( \funct{free}_* \) is a left adjoint functor, it commutes with cokernels and pushouts (it commutes with all colimits). Thus, in both cases, the following square commutes
\[
\begin{tikzcd}
\rep(Q, \vectF)\ar[r, "\mu_d"] & \rep(\mu_d(Q), \vectF) \\
\rep(Q, \pset)\ar[r,"\mu_d"]\ar[u, "\funct{free}_*"] & \rep(\mu_d(Q), \pset)\ar[u, "\funct{free}_*"]
\end{tikzcd}.
\]
In particular the linear (i.e, the upper, in the diagram) \( \mu_d \) maps preserves the (additive) images of the vertical functors.

The final claims follow since the reflection functor \( \mu_d \) is surjective on indecomposables except for the simple representation \( S_d \), which is always \( \pSet \)-realizable.
\end{proof}

In Section~\ref{subsec.d4} we will employ this theorem to transfer \( \Set \)-realizability from certain orientations of quivers of type \( \tilde{D}_4 \) to others.

\section{Additive images consisting of indicator representations}
\label{subsect:only_thin}

In this section, we give a necessary and sufficient criterion for the additive image to only consist of indicator representations.

Throughout this section, we assume that $X$ is a finite category.

\begin{theorem}
\label{thm.onlythin}
Let \( X \) be a finite category. Then every indecomposable additively \( \Set \)-realizable representation is an indicator representation if and only if \( X \) is equivalent to the poset category of a poset that is a disjoint union of
\begin{enumerate}
\item posets with a Hasse diagram of type \( A \),
\item posets with a single maximal element, such that the poset without the maximal element is a disjoint union of posets with a Hasse diagram of type \( A \) with unique minima.
\end{enumerate}
\end{theorem}

Let us first check that for the two families described in the theorem only indicator representations are indecomposable additively \( \Set \)-realizable.

\begin{proof}[Proof of \( \Leftarrow \)]
Suppose $X$ satisfies condition (1). In this case, every indecomposable representation is an indicator representation, and the result follows immediately (see also Proposition~\ref{prop.An}). 

Now, assume $X$ satisfies condition (2). Let $R$ be an indecomposable additively $\Set$-realizable representation, and let $t$ denote the maximal element of $X$. By Lemma~\ref{lem.terminal_is_1d}, we have $\dim R(t) \leq 1$. 

If $\dim R(t) = 0$, the result follows from the case of type $A$ quivers. Hence, we may assume $\dim R(t) = 1$. Consider any component \( C \) of the subposet \( X \setminus \{ t \} \). Each component $C$ is of one of the following two forms:
\[
\circ \to \cdots \to \circ \qquad \text{or} \qquad \circ \leftarrow \circ \leftarrow \cdots \leftarrow \circ \rightarrow \cdots \circ \rightarrow \circ.
\]

Consider the full subcategory of $X$ given as $Y = C \cup \{ t \}$. We claim that the restriction of $R$ to $Y$, denoted $R|_Y$, is indecomposable. Indeed, if $R|_Y \cong R' \oplus R''$ is a non-trivial decomposition, we may assume $\dim R'(t) = 0$. Then, any split monomorphism $R' \hookrightarrow R|_Y$ extends to a split monomorphism $R' \hookrightarrow R$ by setting $R'$ to be $0$ on all vertices outside $C$. This contradicts the indecomposability of $R$, so $R|_Y$ must be indecomposable.

If $C$ is of the left form, or if $C$ is of the right form with $R(s) = 0$ for the minimal element $s$, then $R|_Y$ (after ignoring the vertex $s$) is an indecomposable representation of a type $A$ quiver, and the result follows. It remains to consider the case where $C$ is of the right form and $\dim R(s) > 0$. This reduces to studying a representation of a quiver of type $\tilde{A}$ with a unique minimal and a unique maximal element, where the two maps from the minimal to the maximal vertex coincide since we assume a poset structure. That such a representation is interval-decomposable is well known from standard representation theory, but we include an argument for completeness. To this end, we distinguish two subcases:

\begin{enumerate}
\item If the map $R(s \to t)$ is non-zero, then pick \( m \in R(s) \) not being sent to zero by this map. We obtain a monomorphism from the indicator representation $I$ on $Y$ to $R|_Y$ defined by mapping $1 \in I(v)$ for $v \in C$ to $R(s \to v)(m)$. But since $I$ is an injective representation (as $Y$ has a unique maximal element), this monomorphism splits, and indecomposability of $R|_Y$ implies $R|_Y \cong I$.
\item If the map $R(s \to t)$ is zero, then pick any indecomposable summand \( R' \) of $R|_C$ supported on $s$. Since \( R' \) is an indicator representation it is generated by \( R'(s) = F \), and thus all elements of \( R'(v) \) for $v \in C$ are sent to zero by \( R(v \to t) \), this extends to a direct summand of \( R \). Hence, $R$ cannot be indecomposable in this case, contradicting the assumption. \qedhere
\end{enumerate}
\end{proof}

As a first step towards the converse, we consider the case of a category with only one object, that is, a monoid. Since the discussion of this case is quite independent of the rest of the proof, we present it as a separate lemma here.

\begin{lemma} \label{lemma:indicator_for_monoids}
Let \( X \) be a finite category with one object. Then all indecomposable additively \( \Set \)-realizable representations are indicator representations only if the category has no morphisms apart from the identity.
\end{lemma}

\begin{proof}
For a monoid \( (M, \cdot) \) we always have the regular representation \( R \) in the set \( M \), given explicitly by
\[ M \to \set(M,M) \colon m \mapsto (m \cdot - ). \]
If all additively \( \Set \)-realizable representations are indicator representations, then so are the summands of
$\funct{free} \circ R$. Since there is only one object, there is only a single such indicator representation, which sends every morphism to the identity. Taking direct sums of copies of this representation does not change this behavior: every morphism is still mapped to the identity, only now acting on a larger vector space. Therefore, \( \funct{free} \circ R \) is isomorphic to some vector space with all morphisms mapped to identity. Note however that base changes do not affect identities, and thus \( R \) itself needs to map every morphism to the identity, in other words \( m \cdot n = n \) for all \( m, n \in M \). Choosing \( n = \operatorname{id} \) shows \( m = \operatorname{id} \), and thus \( M = \{ \operatorname{id} \} \) as claimed. 
\end{proof}

\begin{proof}[Proof of \( \Rightarrow \) in Theorem~\ref{thm.onlythin}]
By Lemma~\ref{lemma:indicator_for_monoids}, in conjunction with Corollary~\ref{cor.subcategory},  we can assume that the only morphisms from an object of \( X \) to itself are identities. If there are two objects with morphism between them in both directions then they thus necessarily compose to identities, so the two objects are isomorphic and up to equivalence we can remove one of them. Thus, up to equivalence we can assume \( X \) to not have any cycles of (non-identity) morphisms.

Next we observe that the category given by the Kronecker quiver $\begin{tikzcd} \circ\ar[r, bend left]\ar[r, bend right] & \circ \end{tikzcd}$ has non-indicator additively \( \Set \)-realizable representations, e.g.,

\[
\begin{tikzcd}[sep=7mm,ampersand replacement=\&]
  \{a,b\} \& \& \& \& F^2 \\
 ~\ar[rrrr,mapsto, shorten <= 30pt, shorten >= 30pt, "\funct{free}_*"] \& \&  \& \& ~\\
  \{a\} \arrow[uu,swap, bend right, "a\mapsto b"]\arrow[uu,"a\mapsto a", bend left]  \& \& \&  \& F\arrow[uu, "{\begin{bmatrix}0 \\ 1\end{bmatrix}}", bend left] \arrow[uu, swap, "{\begin{bmatrix}1 \\ 0\end{bmatrix}}", bend right]
\end{tikzcd}
\]
In fact, all preprojective representations of the Kronecker quiver are realizable.
This example can easily be extended to any number of parallel arrows, and using Corollary~\ref{cor.subcategory} also to any other category with at least two parallel arrows between some pair of distinct objects.
It now follows from Corollary~\ref{cor.subcategory} that there is at most one morphism between any two objects in \( X \), that is, \( X \) is a poset category.

Next, it is easy to check that any quiver of type \( D_4 \), except the one with all arrows pointing towards the middle, has a non-indicator  additively \( \Set \)-realizable representation; see  Proposition~\ref{prop:when_is_Dn_realizable}. It follows that
\begin{itemize}
\item No \( x  \in X \) is smaller than \( 3 \) or more pairwise incomparable elements of \( X \).
\item If \( x \in X \) is smaller than \( 2 \) incomparable elements then \( x \) is minimal.
\item If \( x \in X \) is greater than \( 2 \) incomparable elements then \( x \) is maximal.
\end{itemize}
Finally note that if an element is smaller than \( 2 \) incomparable elements, and one of those two in turn is bigger than at least \( 3 \) incomparable elements, then \( X \) has a subposet
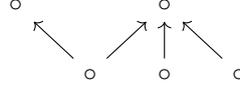
\[ \begin{tikzcd}[column sep=5mm,row sep=5mm]
\circ && \circ \\
& \circ \ar[lu] \ar[ru] & \circ \ar[u] & \circ \ar[lu]
\end{tikzcd} \]
This poset has non-indicator additively \( \Set \)-realizable representations, e.g.,
\[ \begin{tikzcd}[column sep=8mm,row sep=10mm,ampersand replacement=\&]
F \& F^2  \ar[l, "\left({\begin{smallmatrix}1 & 1\end{smallmatrix}}\right)"] \ar[r, "\text{id}"]\&  F^2 \\
\& \& F \ar[u,"\left({\begin{smallmatrix}1 \\ 0 \end{smallmatrix}}\right)"] \& F \ar[lu,swap,"\left({\begin{smallmatrix}0 \\ 1\end{smallmatrix}}\right)"]
\end{tikzcd} \]

Therefore this cannot happen under the assumption that \( X \) does not have non-indicator additively \( \Set \)-realizable representations.

It follows that if there is an element in \( X \) which is smaller than two distinct maxima then \( X \) is of type \( A \) or \( \widetilde{A} \). Considering preprojective representations one easily sees that algebras of type \( \widetilde{A} \) have non-indicator \( \Set \)-realizable representations; e.g., the representation of the Kronecker quiver above easily extends to $\widetilde{A}$; see also Theorem~\ref{thm.Ainf}.

If \( X \) does not have an element which is smaller than two distinct maximal elements, and is connected, then \( X \) has a unique maximal element. Again invoking Corollary~\ref{cor.subcategory} we know that the same results hold for \( X' = X \setminus \{ \text{maximal element} \} \), with the additional condition that the elements smaller than any maximal element of \( X' \) are linearly ordered by the third bullet point above. By the first bullet point, at most two chains of descending elements from distinct maximal elements in $X'$ can intersect, and by the second bullet point, the point of intersection must be at a minimal element. It follows that \( X' \) is a disjoint union of type \( A \) quivers with unique sources.
\end{proof}

We believe that Theorem~\ref{thm.onlythin} should extend to infinite \( X \) for some suitable formulation of what it means for an infinite category to be of type \( A \). (The idea would be that our theorem restricts the possible finite subcategories of any such infinite \( X \), thus giving it little room to be of a different shape.)

\section{Quiver representations} \label{sec.quivers}

In this section, we only consider representations of quivers. Our aim is to provide a reasonably complete answer to Question~\ref{questions} for Dynkin and certain extended Dynkin quivers.

\subsection{Dynkin quivers}

The first point we want to make for Dynkin quivers is that it does not matter if we consider all linear representations or just finite-dimensional ones: By classical theory the indecomposables are finite-dimensional anyway.

\begin{theorem}
Let \( Q \) be a Dynkin quiver. Then the functor \( \funct{free}_* \colon \rep(Q, \Set) \to \rep(Q, \VectF) \) is essentially surjective (additively surjective) if and only if the functor \( \funct{free}_* \colon \rep( Q, \set) \to \rep( Q, \vectF) \) is.
\end{theorem}

\begin{proof}
Since \( \rep( Q, \vectF) \) contains only finitely many indecomposable objects, by \cite{RepresentationTheoryOfArtinAlgebrasII}*{Corollary~4.8} we know that any representation of \( Q \) is a sum of finite-dimensional ones. Now the claim follows immediately, since \( \Set \) has coproducts.
\end{proof}

\subsubsection{Type $A$}

All the indecomposables are indicator representations. The following is therefore immediate from Corollary~\ref{cor.const_essimg}.

\begin{proposition}
\label{prop.An}
Let \( Q \) be a quiver of type \( A_n \), with any orientation. Then the functor
\[ \funct{free}_* \colon \rep( Q, \pset) \to \rep( Q, \vectF) \]
is essentially surjective.
\end{proposition}

\subsubsection{Type $D$}

\begin{proposition} \label{prop:when_is_Dn_realizable}
For a quiver \( Q \) of type \( D_n \), the following are equivalent.
\begin{enumerate}
\item the functor \( \funct{free}^\mathrm{pt} \colon \rep( Q, \pset) \to \rep( Q, \vectF) \)
is essentially surjective;
\item the functor \( \funct{free} \colon \rep( Q, \set) \to \rep( Q, \vectF) \)
is additively surjective;
\item not every arm of the quiver contains an ``ingoing arrow'', that is \( Q \) is not of the form
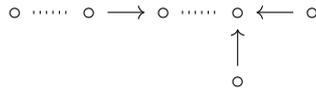
\[ \begin{tikzcd}[sep=5mm]
\circ \ar[r,-,dotted,very thick] & \circ \ar[r] & \circ \ar[r,-,dotted,very thick] & \circ & \circ \ar[l] \\
&&& \circ \ar[u] 
\end{tikzcd} \]
where the arrows in the dotted part may have any orientation.
\end{enumerate}
\end{proposition}

\begin{proof}
(1) implies (2) by Proposition~\ref{prop.same_add_image}.

To show that (3) implies (1), we recall that any indecomposable representation of a quiver of type $D_n$ is either an indicator representation or can be chosen to have the following form \cite[Chapter 8]{schiffler2014quiver}:
\[
\begin{tikzcd}[sep=5mm]
0 \ar[r,-,dotted,very thick] & F \ar[r,-,dotted,very thick,"{\rm id}"] & F \ar[r,-,"A"] & 
F^2 \ar[r,-,dotted,very thick,"{\rm id}"] & F^2 & \ar[l,-,"B"] F \\
&&&& F \ar[u,-,"C"] 
\end{tikzcd}
\]
where the maps $A$, $B$ and $C$ are appropriately chosen projections, inclusions, and diagonal maps ($[1,1]$ or its transpose) and of which exactly one is a diagonal map. Moreover, these maps can be chosen such that the diagonal map appears at any of the three arms. Under the stated assumption, one can thus always choose the diagonal map on an outgoing arrow. In conclusion,  every indecomposable satisfies the condition in Proposition~\ref{prop:mult-basis}.

Finally, (2) implies (3) was shown in Example~\ref{example:D_n_not_realizable}.
\end{proof}

\subsubsection{Type $E$}

Computer experiments\footnote{\url{https://github.com/mbotnan/set-realizable/}}, using Corollary~\ref{cor.criterion_for_add_image} and the field \( F = \Z /2\Z \), gave the following complete list of orientations for which the free functor is additively surjective.
\begin{itemize}
\item[\(E_6\)] 
\begin{itemize}
\item One arm of length two is completely oriented outwards, or
\item the arm of length one is oriented outwards, and at least one additional arrow is oriented outwards.
\end{itemize}
\item[\(E_7\)]
\begin{itemize}
\item The arm of length three is completely oriented outwards, or
\item the arm of length two is completely oriented outwards, and at least one additional arrow is oriented outwards, or
\item the arm of length one is oriented outwards, and in at least one additional arm at least all but one arrow are oriented outwards.
\end{itemize}
\item[\(E_8\)]
\begin{itemize}
\item The arm of length four is completely oriented outwards, and at least one additional arrow is oriented outwards, or
\item the arm of length two is completely oriented outwards, and the arm of length four has at least two outward-oriented arrows, or
\item the arm of length one is oriented outwards, and at least two additional arrows are oriented outwards, at least one of which is on the arm of length two, or
\item the arm of length one is oriented outwards, and at least three arrows on the arm of length four are oriented outwards.
\end{itemize}
\end{itemize}
\subsection{Type \texorpdfstring{$\widetilde{A}$}{A-tilde}} \label{subsec:Atilde}

In this section we let $Q$ be a quiver of type $\widetilde{A}_n$,  \[ Q = 
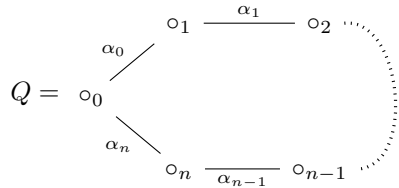
\begin{tikzcd}[sep=5mm]
& \circ_1 \ar[rr,-,"\alpha_1"] && \circ_2 \ar[dd,-,dotted,very thick,bend left=90] \\
\circ_0 \ar[ru,-,"\alpha_0"] \ar[rd,-,swap,"\alpha_n"] \\
& \circ_n \ar[rr,-,swap,"\alpha_{n-1}"] && \circ_{n-1}
\end{tikzcd} \]
The indecomposable representations of $Q$ are either \emph{string} or \emph{band} representations. This is proven in \cite{Butler-Ringel}*{Theorem on page 161} more generally for string algebras, we recall the result here for our special case in the following theorem.

\begin{theorem} \label{thm:reps_of_Atilde}
Let $R \in \rep(Q, \vectF)$. Then $R$ is indecomposable if and only if $R$ is of one of the following two forms.  
\begin{itemize}
\item (String) $R\cong \funct{free}^\mathrm{pt} \circ S_\gamma$ for some $S_\gamma$ defined as follows. 

\begin{itemize}
\item 
Let $\gamma\colon \{1, 2, \ldots, m\}\to Q_0$ be an undirected path moving clockwise around $Q$. (Explicitly, \( \gamma \) is given by adding a constant and considering the result modulo \( n+1 \).) Let 
\[ S_\gamma(v) =  \gamma^{-1}(v) \cup \{ * \} \in \pset_* \]
which is a representation by saying that for $1\leq l<m$ we have $l \mapsto l+1$ or $l+1 \mapsto l$ (depending on the direction of the arrow $\alpha_{\gamma(l)}$, and interpreting the vertex and arrow indices cyclically).

\item
At the ends of the path we declare that if the arrow \( \alpha_{\gamma(1)-1} \) is oriented outward from \( \circ_{{\gamma(1)}}\) then \( S_\gamma(\alpha_{\gamma(1)-1})(1) = * \), and symmetrically if the arrow \( \alpha_{\gamma(m)} \) is oriented outward from \( \circ_{\alpha_{\gamma(m)}}\) then \( S_\gamma(\alpha_{\gamma(m)})(m) = * \).
\end{itemize}

\item (Band) \[R \cong 
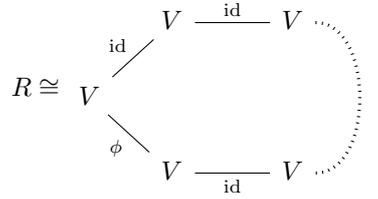
\begin{tikzcd}[sep=5mm]
  & V \ar[rr,-,"{\rm id}"] && V \ar [dd,-,dotted,very thick,bend left=90] \\
  V \ar[ru,-,"{\rm id}"] \ar[rd,-,swap,"\phi"] \\
  & V \ar[rr,-,swap,"{\rm id}"] && V
  \end{tikzcd}, \]
where $\phi\colon V \to V$ is an isomorphism such that the following representation of the 1-loop quiver is indecomposable,
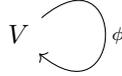
\[
\begin{tikzcd}
V \arrow[out=40,in=-40,loop,"\phi"]
\end{tikzcd}
\]
Equivalently, $V=F[X]/(p(X)^m)$ for a monic irreducible polynomial $p(X)$ and $\phi(q) = X\cdot q$. 
\end{itemize}
\end{theorem}
\begin{remark}
The \emph{companion matrix} $C(q)$ of $q(X) = c_0 + c_1X + \ldots + c_{u-1}X^{u-1}+X^u$ is the $u\times u$ matrix 
\[ \begin{bmatrix} 0 & 0 & \cdots & 0 & -c_0 \\ 1 & 0 & \cdots & 0 & -c_1 \\ 0 & 1  & \cdots & 0 & -c_2 \\
\vdots & \vdots & \vdots & \vdots & \vdots\\
0 & 0  &\cdots & 0 & -c_{u-2}\\ 0 & 0  & \cdots & 1 & -c_{u-1} \\\end{bmatrix}.\]
The basis \( (1, X, X^2, \ldots, X^{m \operatorname{deg}(p) - 1}) \) of \( V \) gives \( C(p^m) \) as matrix representation for \( \phi \).

In the simplest case that $p(X) = X-\lambda$ for some $\lambda\in F$, the basis \( (1, X - \lambda, (X - \lambda)^2, \cdots, (X - \lambda)^{m-1} ) \) of \( V \) gives the $m\times m$ Jordan matrix $J_m(\lambda)$ as matrix representation for \( \phi \) (which is thus similar to $C(p^m)$). 
\end{remark}

By definition, the string representations are in the essential image of $\funct{free}^\mathrm{pt}_*$. We shall prove the following. 

\begin{theorem}\label{thm.Ainf}
Let \( Q \) be a quiver of type \( \tilde A \). An indecomposable \( R \in \rep( Q, \vectF) \) is in the additive image of the functor \( \funct{free}_* \colon \rep( Q, \set ) \to \rep( Q, \vectF ) \) if and only if
\begin{itemize}
\item $R$ is a string representation, or
\item $R$ is a band representation with $V\cong F[X]/(f(X)^m)$ where
\begin{enumerate}
\item $f(X)$ is the minimal polynomial of a root of unity,
\item the number $m$ satisfies
\[ m = \begin{cases} 1 & \text{if ${\rm char}(F) = 0$} \\ {\rm char}(F)^k & \text{if ${\rm char}(F) > 0$}\end{cases},\]
where $k$ is a non-negative integer and ${\rm char}(F)$ denotes the characteristic of $F$, 
\item the isomorphism $\phi\colon V\to V$ in the definition of a band representation is given by $\phi(q) = X\cdot q$.
\end{enumerate}
\end{itemize}
\end{theorem}
This theorem is surprising and illustrates the intrinsic complexity of characterizing the additively $\Set$-realizable representations. Consider for instance the following example.
\begin{example}
Let $F=\Z/2\Z$, $V=(\Z/2\Z)^m$, and
\[\phi = J_m(1) = 
\begin{bmatrix}
 1 & 1 & 0& \cdots & 0 & 0 \\
 0 & 1 & 1 &\cdots & 0 & 0 \\
\vdots & \vdots & \vdots & \ddots & \vdots & \vdots \\
  0 & 0 & 0 & \cdots & 1 &  1\\
  0 & 0 & 0 & \cdots & 0 & 1
  \end{bmatrix}.\]
Then the corresponding band representation is additively $\Set$-realizable if and only if $m = 2^k$ for some integer $k\geq 0$.
\end{example}

By virtue of the following lemma we see that it suffices to consider automorphisms of vector spaces. 
\begin{lemma}
\label{lem.bandrep}
Let $R$ be a band representation, and let $\phi\colon V\to V$ be the isomorphism as in the definition of a band representation. Then $R$ is additively $\Set$-realizable if and only if the following representation $R'$ of the one-loop quiver is additively $\Set$-realizable, 
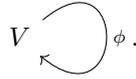
 \[\begin{tikzcd}
V \arrow[out=40,in=-40,loop,"\phi"]
\end{tikzcd}.\]
\end{lemma}

\begin{proof}
A quiver of type \( \widetilde{A} \) can be turned into a one-loop quiver $L$ by iterated edge contractions (see Definition~\ref{def:edge_contraction_quiver}). Denoting the composition of edge contractions by $\funct{c}\colon \widetilde{A}\to L$, the functor $\rep(L, \vectF) \to \rep (\widetilde{A},\vectF)$ given by pre-composition by $\funct{c}$ acts by adding copies of $V$ and identity maps. In particular, the band representation $R$ as depicted in Theorem~\ref{thm:reps_of_Atilde} is of the form $R'\circ \funct{c}$. The claim follows from Theorem~\ref{thm.edgecontraction}.
\end{proof}

%

Let $L$ denote the one-loop quiver. First we observe that a representation $(S,a)\in \rep(L,\set)$, consisting of a single set $S$ and a map $a \colon S \to S$, corresponds to a quiver where each vertex is an element of $S$ and has out-degree one (sometimes called a \emph{functional graph}), as illustrated by the following diagram:

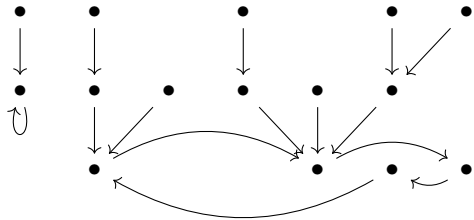
\[\begin{tikzcd}[column sep=5mm]
\bullet\ar[d] & \bullet\ar[d] &  & \bullet\ar[d]&  & \bullet\ar[d]& \bullet\ar[dl]\\
\bullet\ar[loop below] & \bullet\ar[d] & \bullet\ar[dl] & \bullet\ar[dr] & \bullet\ar[d] & \bullet \ar[dl]&\\
 & \bullet\ar[rrr, bend left] &  &  & \bullet\ar[rr, bend left] & \bullet\ar[llll, bend left] & \bullet \ar[l, bend left]
\end{tikzcd}
\]

In particular, every element $i\in S$ ultimately ends up in a cycle. Let $C$ denote the set of elements $i\in S$ for which there exists a $k>1$ such that $a^k(i) = i$.
\begin{lemma}\label{lem.Ainf-inj}
Let $(S,a)\in \rep(L, \set)$. Then, 
\[\funct{free}_*(S,a) \cong \funct{free}_*(C, a|_C) \oplus (R, \psi)\]
for some $(R, \psi)\in \rep(L, \vectF)$ where $\psi$ is nilpotent. 
\end{lemma}

\begin{proof}
We construct a one-sided inverse to the inclusion \( (C, a|_C) \to (S, a) \).

For $c\in C$, let $\operatorname{pre}(c)$ denote the unique $c'\in C$ such that $a(c') = c$. For \( s \in S \), find the minimal number \( m \) such that \( a^m(s) \in C \), and set \( f(s) = \operatorname{pre}^m(a^m(s)) \in C \). Then \( f \) is a morphism \( (S, a) \) to \( (C, a|_C) \), and is left inverse to the inclusion \( (C, a|_C) \to (S, a) \).

Having a left inverse is preserved by any functor, in particular by \( \funct{free}_* \), and thus the short exact sequence
\[ 0 \to \funct{free}_*(C, a|_C) \to \funct{free}_*(S, a) \to (\funct{free}_*(S)/\funct{free}_*(C), \psi) \to 0 \]
splits. Here \( \psi \) is the map induced by \( a \) on the quotient \( \funct{free}_*(S)/\funct{free}_*(C) \). Since \( a \) eventually maps every element of \( S \) to \( C \) it follows that \( \psi \) is nilpotent.
\end{proof}

Since the morphism $\phi\colon V\to V$ in a band representation is an isomorphism, the additively $\Set$-realizable band representations correspond precisely to summands of representations of the form $\funct{free}_*(C,a|_C)$ for $C$ a disjoint union of cycles.

\begin{proposition}\label{prop.Ainf-set}
Let $(S,a)\in \rep(L, \set)$ where $(S,a)$ is a disjoint union of cycles. Then, $\funct{free}_*(S,a)$ decomposes into a direct sum of representations of the form
\[  (F[X] / (X^n-1), ``\cdot X"). \] 
\end{proposition}
\begin{proof}Up to isomorphism, $(S,a)$ can be written as a disjoint union of representations of the form \( ( \Z /n\Z , ``+1") \). It is immediate that
\[\funct{free}_*(\mathbb Z /n\Z, ``+1") \iso (F[X] / (X^n-1), ``\cdot X"). \qedhere \]
\end{proof}

We are now ready to complete the proof of Theorem~\ref{thm.Ainf}.
\begin{proof}[Proof of Theorem~\ref{thm.Ainf}]
From the definition of band representations, Lemma~\ref{lem.bandrep}, Lemma~\ref{lem.Ainf-inj}, and Proposition~\ref{prop.Ainf-set}, it suffices to describe the indecomposable summands of
\[(F[X] / (X^n-1), ``\cdot X").\]
If the characteristic \( p \) of \( F \) is positive, we factor \( n = p^{\alpha} m \) with \( p \) not dividing \( m \). Note that \( X^n - 1 = (X^m - 1)^{p^{\alpha}} \), and that \( X^m - 1 \) is separable, i.e., does not have any double factors in its factor decomposition.

On the other hand, if the characteristic of \( F \) is zero then \( X^n - 1 \)  is separable. To make the notation consistent with the previous notation, we set \( p = 1 \) and \( m = n \) in this case.

Independent of the characteristic it now follows from the Chinese Remainder Theorem that
\[ F[X] / (X^n-1) \cong \bigoplus_{f(X) \text{ irreducible factor of } X^m-1} F[X]/(f(X)^{p^{\alpha}}). \]
To conclude the proof, we note that irreducible factors of polynomials of the form \( X^m - 1 \) are precisely minimal polynomials of roots of unity. 
\end{proof}

\subsection{Type \texorpdfstring{$\tilde D_4$}{D\_4-tilde}}
\label{subsec.d4}
If $Q$ is a quiver of type $\tilde D_4$ with at most one outgoing arrow, then $Q$ is a tree quiver and by Corollary~\ref{cor.treequiver} there are only finitely many indecomposables in the additive image of $\funct{free}_*$. For two ingoing and two outgoing arrows we observe the following.

\begin{proposition} \label{prop.D4tilde_via_A3tilde}
Let $Q$ be a quiver of type $\tilde D_4$ with exactly two outgoing arrows. Then, there are infinitely many indecomposables (not) contained in the additive image of $\funct{free}_*$. 
\end{proposition}

\begin{proof}
Consider the following two indecomposable representations, where $R'$ is obtained from $R$ by leaving out the center vertex: \[
R=\begin{tikzcd}[sep=12mm, ampersand replacement = \&]
\& F^m\&  \\
F^m \ar[r, "{\left[ \begin{smallmatrix} I_m \\ 0 \end{smallmatrix} \right]}"] \& F^{2m}\ar[u,swap, "{\left[ \begin{smallmatrix} I_m & A \end{smallmatrix} \right]}"] \ar[r, "{\left[ \begin{smallmatrix} I_m & I_m \end{smallmatrix} \right]}"] \& F^m \\
\& F^m \ar[u, "{\left[ \begin{smallmatrix} 0 \\ I_m \end{smallmatrix} \right]}"]  \& \end{tikzcd} 
\qquad R' =
\begin{tikzcd}[sep=12mm, ampersand replacement = \&]
\& F^m\&  \\
F^m\ar[ur, "I_m"]\ar[rr, near start, "I_m"]  \& \& F^m \\
\& F^m \ar[ur, "I_m"] \ar[uu, near start, "A"]  \& \end{tikzcd},
\]
and $A$ is the companion matrix of $p^s$, where $p(X)\neq X, X-1$ is a monic and irreducible polynomial, and $s\geq 1$ is an integer. Note that $R$ is additively $\Set$-realizable if and only if $R'$ is additively $\Set$-realizable. The forward implication is obvious. The converse follows from Proposition~\ref{prop.Lan_is_realizable} and the observation that $R$ is the left Kan extension of $R'$ along the inclusion appending the central vertex. The result follows from Theorem~\ref{thm.Ainf} as $R'$ is a band representation of $\tilde{A}_3$.
\end{proof}

Before we look more closely at which exact representations are in the additive image of \( \funct{free}_* \) we need a brief review of the representation theory of \( \tilde{D}_4 \), in particular aspects of Auslander--Reiten theory.

As for any hereditary algebra, there are preprojective and preinjective representations whose dimension vectors can easily combinatorially be computed (see \cite{AR}). Moreover, indecomposable preprojective and preinjective representations are uniquely determined by their dimension vectors \cite{AR}*{VIII, Corollary~2.3}.

In addition, there are regular representations. These come in a family of tubes, that is Auslander-Reiten components on which the Auslander-Reiten translation has a finite order. This finite order is called the rank of the tube. For type $\widetilde{D}_4$, of these tubes, 3 have rank 2 while the others have rank 1. See \cite{SS2}*{XIII.2} and in particular Table~2.6 therein. (Note that that book uses algebraically closed fields throughout, but the result holds for arbitrary fields. In the proof of our main result here, Theorem~\ref{thm:D4_three_outgoing}, we will however reduce to the case of an algebraically closed field to be able to cite the source.)

\begin{remark} \label{rem:T0ogTinfty}
Let $Q$ be a quiver of type $\tilde D_4$ with exactly two outgoing arrows. The same argument as in the proof of Proposition~\ref{prop.D4tilde_via_A3tilde} also shows that all representations in the two tubes of rank \( 2 \) denoted \( \mathcal{T}_0^{\Delta(\tilde D_4)} \) and \( \mathcal{T}_\infty^{\Delta(\tilde D_4)} \) in \cite{RepresentationTheoryOfArtinAlgebrasII}*{XIII.2, Table~2.6} are additively \( \Set \)-realizable: They are obtained as left Kan-extensions of string representations of \( \tilde A_3 \).
\end{remark}

\begin{proposition} \label{prop.preproj_and_preinj_are_induced}
Let $Q$ be a quiver of type $\tilde D_4$ with exactly two outgoing arrows. Then all preprojective and preinjective representations are in the essential image of \( \funct{free}_* \).
\end{proposition}

\begin{proof}
This is done by explicit classification. The pointed set realizations for the preprojectives and preinjectives can be found in Figures~\ref{table.preproj_D4tilde} and \ref{table.preinj_D4tilde}, respectively. Note that it is easy (but tiresome) to check that these in fact do realize all preprojective (resp.\ preinjective) representations: The preprojectives (and preinjectives) are the unique indecomposables with the correct dimension vector. Thus it suffices to check that \( \funct{free}_* \) of the representations in the list all have endomorphism ring \( F \) (this is the bad part), and that they have the correct dimension vectors.
\end{proof}

\begin{example}
\label{ex.D4.preproj.indec}
Let us illustrate the type of calculations it takes to show that the endomorphism ring of \( \funct{free}_* \) of the representations in Figures~\ref{table.preproj_D4tilde} and \ref{table.preinj_D4tilde} have endomorphism ring \( F \). To that end, we look at the first representation in Figure~\ref{table.preproj_D4tilde}. Linearizing it, we obtain the vector space representation
\[ \begin{tikzcd}[ampersand replacement = \&]
F^n \ar[rd, "{\left[ \begin{smallmatrix} I_n \\ 0 \end{smallmatrix} \right]}"] \&\& F^n \\
\& F^{2n} \ar[ru, "{\left[ I_n \; I_n \right]}"] \ar[rd, "{\left[ \begin{smallmatrix} I_n \\ \scriptscriptstyle 000 \end{smallmatrix} \begin{smallmatrix} \scriptscriptstyle 000 \\ I_n \end{smallmatrix} \right]}"] \& \\
F^n \ar[ru, "{\left[ \begin{smallmatrix} 0 \\ I_n \end{smallmatrix} \right]}"] \&\& F^{n+1}
\end{tikzcd} \]
where \( \left[ \begin{smallmatrix} I_n \\ \scriptscriptstyle 000 \end{smallmatrix} \right] \) and \( \left[ \begin{smallmatrix} \scriptscriptstyle 000 \\ I_n \end{smallmatrix} \right] \) denote the matrix \( I_n \) with an extra row of zeros added at the bottom and the top, respectively.

An automorphism of this representation is given by five matrices in the five positions, commuting with the structure maps. That is, matrices \( A, B, D \in F^{n \times n}, C \in F^{2n \times 2n}, E \in F^{(n+1) \times (n+1)} \) such that
\begin{align*}
\left[ \begin{smallmatrix} I_n \\ 0 \end{smallmatrix} \right] A & = C \left[ \begin{smallmatrix} I_n \\ 0 \end{smallmatrix} \right] & \left[ \begin{smallmatrix} 0 \\ I_n \end{smallmatrix} \right] B & = C \left[ \begin{smallmatrix} 0 \\ I_n \end{smallmatrix} \right] &
\left[ I_n \; I_n \right] C & = D \left[ I_n \; I_n \right] & \left[ \begin{smallmatrix} I_n \\ \scriptscriptstyle 000 \end{smallmatrix} \begin{smallmatrix} \scriptscriptstyle 000 \\ I_n \end{smallmatrix} \right] C & = E \left[ \begin{smallmatrix} I_n \\ \scriptscriptstyle 000 \end{smallmatrix} \begin{smallmatrix} \scriptscriptstyle 000 \\ I_n \end{smallmatrix} \right].
\end{align*}
The first two equations imply that \( C = \left[ \begin{smallmatrix} A & 0 \\ 0 & B \end{smallmatrix} \right] \). Then the third equation implies that \( A = B = D \). Now we turn to the final equation, and see that
\[ \left[ \begin{smallmatrix} A \\ \scriptscriptstyle 000 \end{smallmatrix} \begin{smallmatrix} \scriptscriptstyle 000 \\ A \end{smallmatrix} \right] = \left[ E_{\text{without last column}} \; E_{\text{without first column}} \right]. \]
One can see that this is only possible if both \( A \) and \( E \) are diagonal matrices with constant diagonals. This means exactly that the endomorphism is given by multiplication by a scalar.
\end{example}

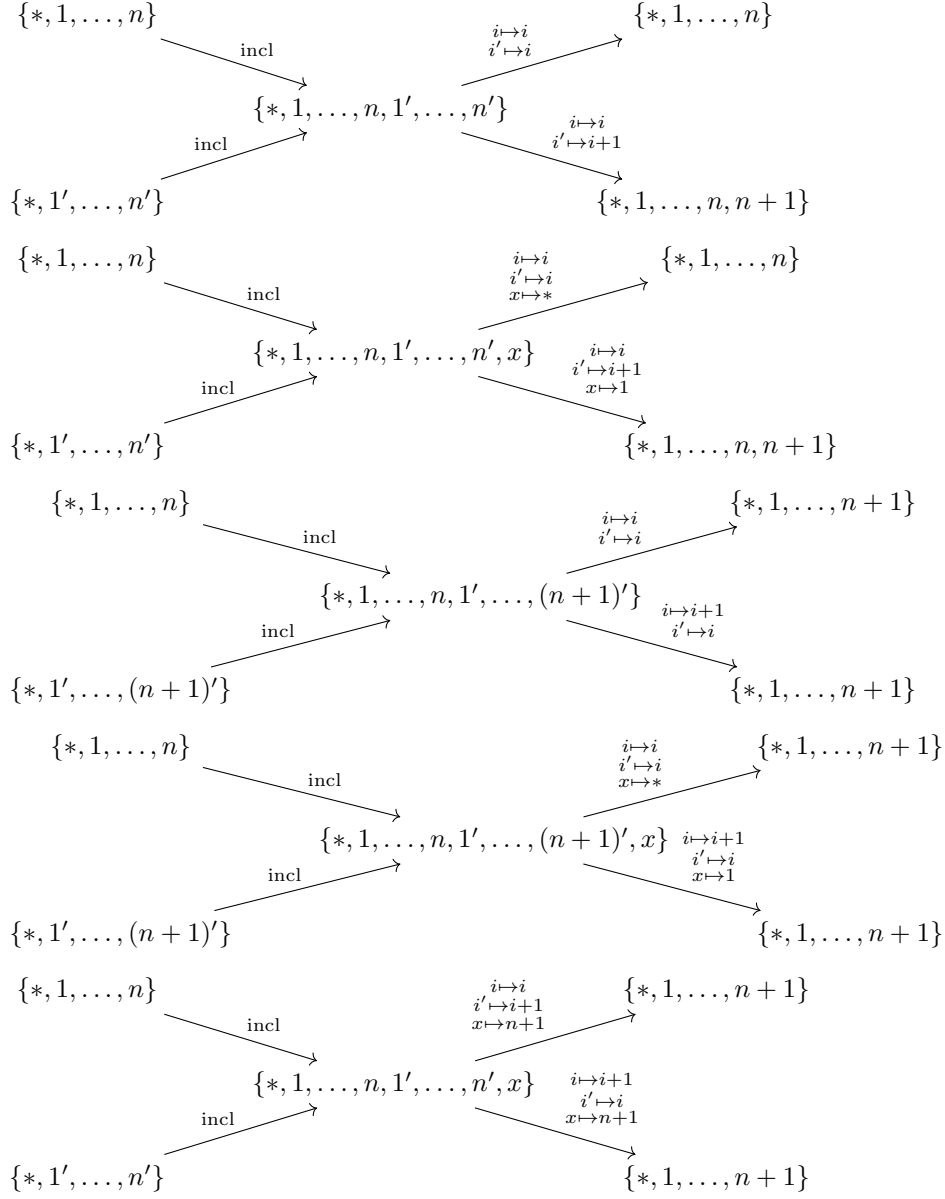
\begin{figure}
\begin{align*}
& \begin{tikzcd}[ampersand replacement = \&]
  \{ *, 1, \ldots, n \} \ar[rd, "\text{incl}"] \&\& \{ *, 1, \ldots, n \} \\
  \& \{ *, 1, \ldots, n, 1', \ldots, n' \} \ar[ru, "\begin{smallmatrix} i \mapsto i \\ i' \mapsto i \end{smallmatrix}"] \ar[rd, "\begin{smallmatrix} i \mapsto i \\ i' \mapsto i+1 
  \end{smallmatrix}"] \& \\
  \{ *, 1', \ldots, n' \} \ar[ru, "\text{incl}"] \&\& \{ *, 1, \ldots, n, n+1 \}
  \end{tikzcd} \\
& \begin{tikzcd}[ampersand replacement = \&]
  \{ *, 1, \ldots, n \} \ar[rd, "\text{incl}"] \&\& \{ *, 1, \ldots, n \} \\
  \& \{ *, 1, \ldots, n, 1', \ldots, n', x \} \ar[ru, "\begin{smallmatrix} i \mapsto i \\ i' \mapsto i \\ x \mapsto * \end{smallmatrix}"] \ar[rd, "\begin{smallmatrix} i \mapsto i \\ i' \mapsto i+1 \\ x \mapsto 1 \end{smallmatrix}"] \& \\
  \{ *, 1', \ldots, n' \} \ar[ru, "\text{incl}"] \&\& \{ *, 1, \ldots, n, n+1 \}
  \end{tikzcd} \\
& \begin{tikzcd}[ampersand replacement = \&]
  \{ *, 1, \ldots, n \} \ar[rd, "\text{incl}"] \&\& \{ *, 1, \ldots, n+1 \} \\
  \& \{ *, 1, \ldots, n, 1', \ldots, (n+1)' \} \ar[ru, "\begin{smallmatrix} i \mapsto i \\ i' \mapsto i \end{smallmatrix}"] \ar[rd, "\begin{smallmatrix} i \mapsto i+1 \\ i' \mapsto i \end{smallmatrix}"] \& \\
  \{ *, 1', \ldots, (n+1)' \} \ar[ru, "\text{incl}"] \&\& \{ *, 1, \ldots, n+1 \}
  \end{tikzcd} \\
& \begin{tikzcd}[ampersand replacement = \&]
  \{ *, 1, \ldots, n \} \ar[rd, "\text{incl}"] \&\& \{ *, 1, \ldots, n+1 \} \\
  \& \{ *, 1, \ldots, n, 1', \ldots, (n+1)', x \} \ar[ru, "\begin{smallmatrix} i \mapsto i \\ i' \mapsto i \\ x \mapsto * \end{smallmatrix}"] \ar[rd, "\begin{smallmatrix} i \mapsto i+1 \\ i' \mapsto i \\ x \mapsto 1 \end{smallmatrix}"] \& \\
  \{ *, 1', \ldots, (n+1)' \} \ar[ru, "\text{incl}"] \&\& \{ *, 1, \ldots, n+1 \}
  \end{tikzcd} \\
& \begin{tikzcd}[ampersand replacement = \&]
  \{ *, 1, \ldots, n \} \ar[rd, "\text{incl}"] \&\& \{ *, 1, \ldots, n+1 \} \\
  \& \{ *, 1, \ldots, n, 1', \ldots, n', x \} \ar[ru, "\begin{smallmatrix} i \mapsto i \\ i' \mapsto i+1 \\ x \mapsto n+1 \end{smallmatrix}"] \ar[rd, "\begin{smallmatrix} i \mapsto i+1 \\ i' \mapsto i \\ x \mapsto n+1 \end{smallmatrix}"] \& \\
  \{ *, 1', \ldots, n' \} \ar[ru, "\text{incl}"] \&\& \{ *, 1, \ldots, n+1 \}
  \end{tikzcd} \\
\end{align*}
\caption{Pointed set realizations for the preprojective representations of $\tilde{D}_4$ with two ingoing and two outgoing arrows; see Proposition~\ref{prop.preproj_and_preinj_are_induced}. The linearized representation obtained from the top-most diagram is shown to be indecomposable in Example~\ref{ex.D4.preproj.indec}. }
\label{table.preproj_D4tilde}
\end{figure}

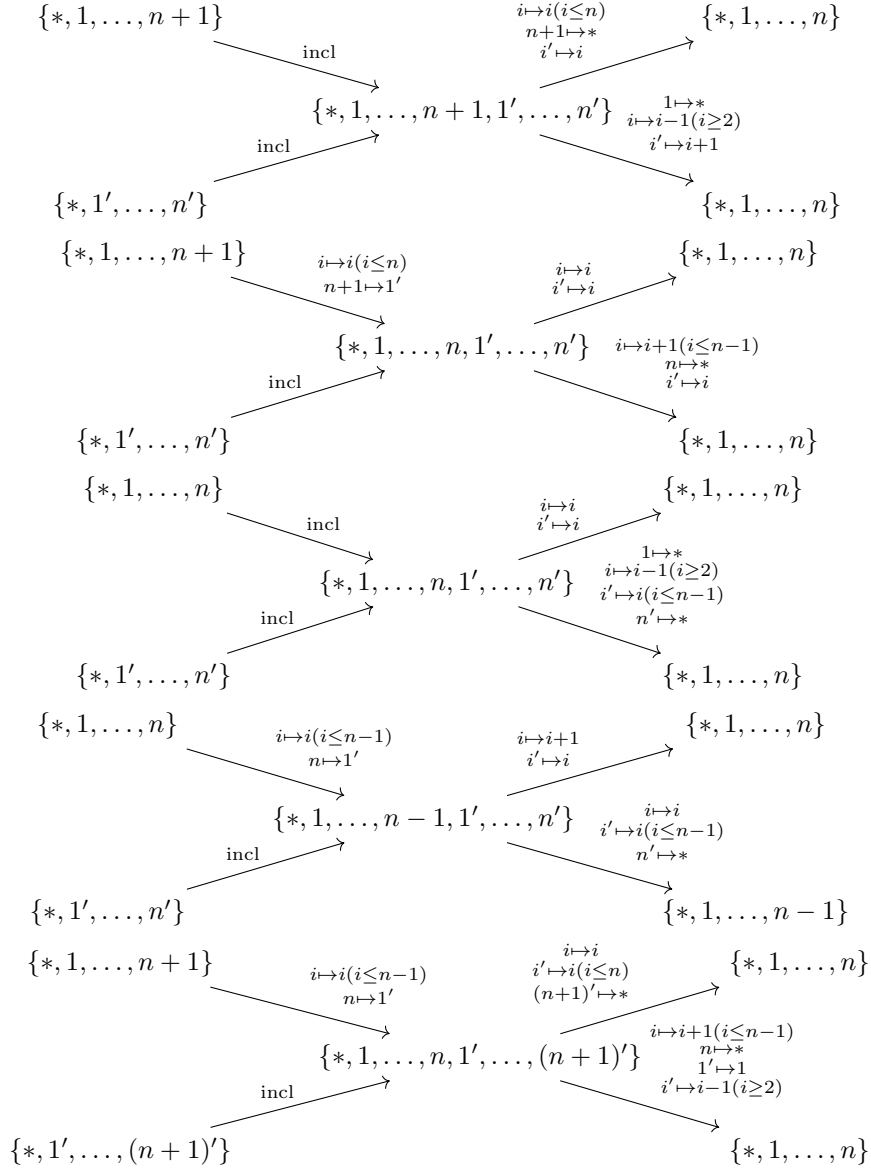
\begin{figure}
\begin{tikzcd}[ampersand replacement = \&]
  \{ *, 1, \ldots, n+1 \} \ar[rd, "\text{incl}"] \&\& \{ *, 1, \ldots, n \} \\
  \& \{ *, 1, \ldots, n+1, 1', \ldots, n' \} \ar[ru, "\begin{smallmatrix} i \mapsto i (i \leq n) \\ n+1 \mapsto * \\ i' \mapsto i \end{smallmatrix}"] \ar[rd, "{\begin{smallmatrix} 1 \mapsto * \\ i \mapsto i-1 (i \geq 2) \\ i' \mapsto i+1 
  \end{smallmatrix}}"] \& \\
  \{ *, 1', \ldots, n' \} \ar[ru, "\text{incl}"] \&\& \{ *, 1, \ldots, n \}
  \end{tikzcd}
\begin{tikzcd}[ampersand replacement = \&]
  \{ *, 1, \ldots, n+1 \} \ar[rd, "\begin{smallmatrix} i \mapsto i (i \leq n) \\ n+1 \mapsto 1' \end{smallmatrix}"] \&\& \{ *, 1, \ldots, n \} \\
  \& \{ *, 1, \ldots, n, 1', \ldots, n' \} \ar[ru, "\begin{smallmatrix} i \mapsto i \\ i' \mapsto i \end{smallmatrix}"] \ar[rd, "{\begin{smallmatrix} i \mapsto i+1 (i \leq n-1) \\ n \mapsto * \\ i' \mapsto i 
  \end{smallmatrix}}"] \& \\
  \{ *, 1', \ldots, n' \} \ar[ru, "\text{incl}"] \&\& \{ *, 1, \ldots, n \}
  \end{tikzcd}
\begin{tikzcd}[ampersand replacement = \&]
  \{ *, 1, \ldots, n \} \ar[rd, "\text{incl}"] \&\& \{ *, 1, \ldots, n \} \\
  \& \{ *, 1, \ldots, n, 1', \ldots, n' \} \ar[ru, "\begin{smallmatrix} i \mapsto i \\ i' \mapsto i \end{smallmatrix}"] \ar[rd, "{\begin{smallmatrix} 1 \mapsto * \\ i \mapsto i-1 (i \geq 2) \\ i' \mapsto i (i \leq n-1) \\ n' \mapsto * 
  \end{smallmatrix}}"] \& \\
  \{ *, 1', \ldots, n' \} \ar[ru, "\text{incl}"] \&\& \{ *, 1, \ldots, n \}
  \end{tikzcd}
\begin{tikzcd}[ampersand replacement = \&]
  \{ *, 1, \ldots, n \} \ar[rd, "\begin{smallmatrix} i \mapsto i (i \leq n-1) \\ n \mapsto 1' \end{smallmatrix}"] \&\& \{ *, 1, \ldots, n \} \\
  \& \{ *, 1, \ldots, n-1, 1', \ldots, n' \} \ar[ru, "\begin{smallmatrix} i \mapsto i+1 \\ i' \mapsto i \end{smallmatrix}"] \ar[rd, "{\begin{smallmatrix} i \mapsto i \\ i' \mapsto i (i \leq n-1) \\ n' \mapsto * \end{smallmatrix}}"] \& \\
  \{ *, 1', \ldots, n' \} \ar[ru, "\text{incl}"] \&\& \{ *, 1, \ldots, n-1 \}
  \end{tikzcd}
\begin{tikzcd}[ampersand replacement = \&]
  \{ *, 1, \ldots, n+1 \} \ar[rd, "\begin{smallmatrix} i \mapsto i (i \leq n-1) \\ n \mapsto 1' \end{smallmatrix}"] \&\& \{ *, 1, \ldots, n \} \\
  \& \{ *, 1, \ldots, n, 1', \ldots, (n+1)' \} \ar[ru, "\begin{smallmatrix} i \mapsto i \\ i' \mapsto i (i \leq n) \\ (n+1)' \mapsto * \end{smallmatrix}"] \ar[rd, "{\begin{smallmatrix} i \mapsto i+1 (i \leq n-1) \\ n \mapsto * \\ 1' \mapsto 1 \\ i' \mapsto i-1 (i \geq 2) \end{smallmatrix}}"] \& \\
  \{ *, 1', \ldots, (n+1)' \} \ar[ru, "\text{incl}"] \&\& \{ *, 1, \ldots, n \}
  \end{tikzcd}   
\caption{Pointed set realizations for the preinjective representations of $\tilde{D}_4$ with two ingoing and two outgoing arrows; see Proposition~\ref{prop.preproj_and_preinj_are_induced}.}
\label{table.preinj_D4tilde}
\end{figure}

Now let us focus on the case where \( Q \) has at least three outgoing arrows. We prove the following.

\begin{theorem} \label{thm:D4_three_outgoing}
Let $Q$ be the quiver of type $\tilde{D}_4$ with three or four outgoing arrows, and let $F$ be an algebraic extension of a finite field. Then the functor $\funct{free}_* \colon \rep(Q, \set)\to \rep(Q, \vectF)$ is additively surjective.
\end{theorem}

\begin{proof}
Let us first observe that we may assume the field \( F \) to be algebraically closed: Recall that by Theorem~\ref{thm.fieldext} a representation \( R \) is in the additive image of \( \funct{free}_* \) if and only if \( R \otimes_F \widehat{F} \) is, where \( \widehat{F} \) denotes the algebraic closure of \( F \). It follows that if we can show the theorem for \( F \) being algebraically closed then it also holds for smaller fields.

In view of Theorem~\ref{thm.APR_preserves_induced} it suffices to consider the case of three outgoing arrows. The same theorem, in conjunction with Proposition~\ref{prop.preproj_and_preinj_are_induced}, shows that all preprojective and preinjective representations are in the additive image of \( \funct{free}_* \). Thus it remains to consider regular representations.

We know explicitly the simple regular representations from \cite{RepresentationTheoryOfArtinAlgebrasII}*{XIII.2, Table 2.6}, by applying a reflection functor.

The simple regulars in the three tubes of rank \( 2 \) are
\[ \begin{tikzcd}
& 0 & \\
0 \ar[r] & F \ar[r] \ar[u] \ar[d] & F \\
& 0 & 
\end{tikzcd}
\qquad \text{ and } \qquad 
\begin{tikzcd}
& F & \\
F \ar[r] & F \ar[r] \ar[u] \ar[d] & 0 \\
& F & 
\end{tikzcd} \]
and versions of these obtained by permuting the outgoing arrows. When applying a reflection functor to translate from the case of two outgoing arrows to three outgoing arrows, these three tubes are the images of \( \mathcal{T}_0^{\Delta(\tilde{D}_4)} \), \( \mathcal{T}_\infty^{\Delta(\tilde{D}_4)} \), and \( \mathcal{T}_1^{\Delta(\tilde{D}_4)} \) from \cite{RepresentationTheoryOfArtinAlgebrasII}*{XIII.2, Table 2.6}. In particular the representations in two of them are additively \( \Set \)-realizable by Remark~\ref{rem:T0ogTinfty}. But then so is the third one by the symmetry of the situation.

There are three tubes of rank 2, which only differ by permutation of the three outgoing arrows. These can be seen to be \( \pSet_* \)-realizable in 3 steps:
Firstly, for \( \widetilde{A}_3 \) with alternating orientation there are two tubes of rank 2, and all representations in these tubes are string representations, hence \( \pSet_* \)-realizable.
Secondly, a comparison as in the proof of \ref{prop.D4tilde_via_A3tilde} turns these into two tubes of rank 2 for \( \widetilde{D}_4 \) with two ingoing and two outgoing arrows.
Thirdly, we apply Theorem~\ref{thm.APR_preserves_induced} to see that two of our tubes of rank 2 consist of \( \pSet_* \)-realizable representations. But then so does the third, because they only differ by permutation of the outgoing arrows.

Finally, it remains to consider the homogeneous tubes, that is tubes of rank 1. The representations in these tubes are of the form
\[ R = \left[ \begin{tikzcd}[sep=12mm, ampersand replacement = \&]
  \& F^m\&  \\
  F^m \ar[r, "{\left[ \begin{smallmatrix} I_m \\ 0 \end{smallmatrix} \right]}"] \& F^{2m} \ar[u,swap, "{\left[ \begin{smallmatrix} I_m & 0 \end{smallmatrix} \right]}"] \ar[r, "{\left[ \begin{smallmatrix} I_m & I_m \end{smallmatrix} \right]}"] \ar[d, "{\left[ \begin{smallmatrix} I_m & J_m(\lambda) \end{smallmatrix} \right]}"] \& F^m \\
  \& F^m \& \end{tikzcd} \right] \]
where \( J_m(\lambda) \) is a Jordan block of size \( m \) with some eigenvalue \( \lambda \not\in \{ 0,1 \} \).

Applying the base change to the central vector space whose matrix is given in block form by \( \left[ \begin{smallmatrix} I_m & -\lambda\cdot I_m \\ 0_m & I_m \end{smallmatrix} \right] \) gives
\[ \begin{tikzcd}[sep=12mm, ampersand replacement = \&]
  \& F^m \&  \\
  F^m \ar[r, "{\left[ \begin{smallmatrix} I_m \\ 0 \end{smallmatrix} \right]}"] \& F^{2m} \ar[u,swap, "{\left[ \begin{smallmatrix} I_m & - \lambda I_m \end{smallmatrix} \right]}"] \ar[rr, "{\left[ \begin{smallmatrix} I_m & (1 - \lambda) I_m \end{smallmatrix} \right]}"] \ar[d, "{\left[ \begin{smallmatrix} I_m & J_m(0) \end{smallmatrix} \right]}"] \&\& F^m \\
  \& F^m \& \end{tikzcd} \]
Note that now we have at most one non-zero entry in each column of the matrices, and thus the representation is additively \( \Set \)-realizable by Corollary~\ref{cor.finite_cat_is_enough}.
\end{proof}

\begin{remark}
For the case of four outgoing arrows we can also immediately use the classification given in \cite{medina2004four} --- their result concerns four ingoing arrows, so we need to apply vector space duality, equivalently transpose all matrices.

It is clear that for all representations that are not of type $0$ (following the notation of \cite{medina2004four}), every column of a transposed matrix is either identically zero, or equal to $1$ at exactly one index. In particular, the representations have a coherent basis and are therefore in the essential image of $\funct{free}_*$ by Proposition~\ref{prop:mult-basis}. The only remaining case is the family of representations of the form
\[\begin{tikzcd}[sep=16mm, ampersand replacement = \&]
\& F^m\&  \\
F^m  \& \ar[l, "{\left[ \begin{smallmatrix} I_m & 0 \end{smallmatrix} \right]}"]F^{2m}\ar[u,swap, "{\left[ \begin{smallmatrix} I_m & J_m(\lambda) \end{smallmatrix} \right]}"] \ar[r, "{\left[ \begin{smallmatrix} I_m & I_m \end{smallmatrix} \right]}"]\ar[d, "{\left[ \begin{smallmatrix} 0 & I_m \end{smallmatrix} \right]}"] \& F^m \\
\& F^m   \& \end{tikzcd}.\]
where $J_m(\lambda)$ is the $m\times m$ Jordan block with eigenvalue $\lambda\in F$ and $\lambda\neq 0,1$, for $F$ algebraically closed as in the proof of Theorem~\ref{thm:D4_three_outgoing}. As before, applying the base change to the central vector space whose matrix is given in block form by $\left[ \begin{smallmatrix} I_m & -\lambda\cdot I_m \\ 0_m & I_m \end{smallmatrix} \right]$
gives
\[\begin{tikzcd}[sep=16mm, ampersand replacement = \&]
\& F^m\&  \\
F^m \& \ar[l, "{\left[ \begin{smallmatrix} I_m & -\lambda\cdot I_m \end{smallmatrix} \right]}"] F^{2m}\ar[u,swap, "{\left[ \begin{smallmatrix} I_m & J_m(0) \end{smallmatrix} \right]}"] \ar[r, "{\left[ \begin{smallmatrix} I_m & (1-\lambda)I_m \end{smallmatrix} \right]}"] \ar[d, "{\left[ \begin{smallmatrix} I_m & 0 \end{smallmatrix} \right]}"] \& F^m \\
\& F^m  \& \end{tikzcd}.\]
Such representations are additively $\Set$-realizable by Corollary~\ref{cor.finite_cat_is_enough}.
\end{remark}

\section{Finite grids}
\label{sec.finitegrids}

A finite grid is a poset of the form
\[ \{ 1, \ldots, n_1 \} \times \{ 1, \ldots, n_2 \} \times \cdots \times \{ 1, \ldots, n_d \} \]
equipped with the coordinatewise partial order, where each factor is totally ordered. 
\begin{corollary} \label{cor:finite_grods_of_infinite_type}
For a finite grid the following are equivalent:
\begin{enumerate}
\item There are infinitely many indecomposable representations.
\item There are infinitely many additively \( \Set \)-realizable indecomposable representations.
\item There are infinitely many indecomposable representations which are not \( \Set \)-realizable.
\end{enumerate}
\end{corollary}

\begin{proof}
Clearly (2) implies (1), and (3) implies (1). Thus it suffices to show that (1) implies the other two statements. Further, by Proposition~\ref{prop.Lan_is_realizable}, it suffices to consider the minimal grids having infinitely many indecomposable representations. These are
\[ \{ 1,2,3\}^2 \text{, } \{1,2\} \times \{1,2,3,4,5\} \text{, and } \{1,2\}^3. \]
The minimality of the planar grids follows directly from the summary in \cite[Theorem 1.3]{bauer2020cotorsion}. Clearly, any non-planar grid must contain the third poset listed above; it follows immediately from the discussion below that this poset admits infinitely many indecomposable representations.

For the first and last of these, we immediately find full subcategories which are quivers of type \( \tilde{D}_4 \) (with two ingoing and two outgoing arrows) and \( \tilde{A}_5 \), as illustrated in the following pictures.
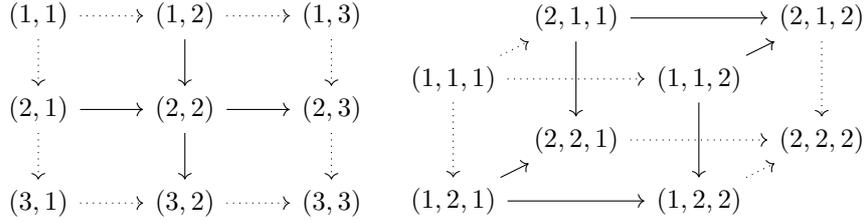
\[ \begin{tikzcd}
(1,1) \ar[r, dotted] \ar[d, dotted] & (1,2) \ar[r,dotted] \ar[d] & (1,3) \ar[d,dotted] \\
(2,1) \ar[r] \ar[d, dotted] & (2,2) \ar[r] \ar[d] & (2,3) \ar[d,dotted] \\
(3,1) \ar[r, dotted] & (3,2) \ar[r,dotted] & (3,3) 
\end{tikzcd}
\quad
\begin{tikzcd}[sep=2mm]
& (2,1,1) \ar[rr] \ar[dd] && (2,1,2) \ar[dd,dotted]  \\
(1,1,1) \ar[rr,dotted] \ar[dd,dotted] \ar[ru,dotted] && (1,1,2) \ar[dd] \ar[ru] \\
& (2,2,1) \ar[rr,dotted] && (2,2,2) \\
(1,2,1) \ar[rr] \ar[ru] && (1,2,2) \ar[ru,dotted] 
\end{tikzcd}
\]
Thus we get the result from Proposition~\ref{prop.D4tilde_via_A3tilde} and Theorem~\ref{thm.Ainf}, respectively.

For \( \{ 1,2 \} \times \{1,2,3,4,5 \} \) we consider representations vanishing in the first and last vertex, and such that the structure maps indicated in the following picture are identities.
\[ \begin{tikzcd}
0 \ar[r] \ar[d] & R_{(1,2)} \ar[r] \ar[d] & R_{(1,3)} \ar[r,equal] \ar[d,equal] & R_{(1,4)} \ar[r] \ar[d] & R_{(1,5)} \ar[d] \\
R_{(2,1)} \ar[r] & R_{(2,2)} \ar[r,equal] & R_{(2,3)} \ar[r] & R_{(2,4)} \ar[r] & 0
\end{tikzcd} \]
Note that the functors \( \funct{free}^\mathrm{pt}_* \) and \( \funct{forget}^\mathrm{pt}_* \) between vector space representations and pointed set representations preserve this subcategory of representations. Thus we can discuss additive \( \Set \)-realizability purely within this subcategory.

Now we can observe that this subcategory is equivalent to \( \tilde{D}_4 \) with two ingoing and two outgoing vertices, so the proof is complete by Proposition~\ref{prop.D4tilde_via_A3tilde}.
\end{proof}

We also have the following counterpart to Corollary~\ref{cor:finite_grods_of_infinite_type}.

\begin{observation} \label{obs:repfinite_grids}
For any grid of finite representation type, all representations are in the essential image of \( \funct{free}_* \).
\end{observation}

\begin{proof}
The key observation here is that it suffices to check the unique largest representation finite grid: \( \{1,2\} \times \{1,2,3,4\} \). For this grid one easily computes the Auslander--Reiten quiver. See for instance \cite{EscolarHiraoka}*{Figure~17}. Most indecomposables are indicator representations.

Many indecomposables have a region of \( 2 \)-s in their dimension vector, and structure maps are given by extending by identities the representations
\[ \begin{tikzcd}
  && F && F \ar[rd,"{\left[ \begin{smallmatrix} 1 \\ 0 \end{smallmatrix} \right]}"] && \\
  F \ar[r,"{\left[ \begin{smallmatrix} 1 \\ 0 \end{smallmatrix} \right]}"] & F \ar[ru,"{[1 \; 0]}"] \ar[rd,"{[1 \; 1]}"] && \text{ or } && F \ar[r,"{[1 \; 1]}"] & F \\
  && F && F \ar[ru,"{\left[ \begin{smallmatrix} 0 \\ 1 \end{smallmatrix} \right]}"] &&
\end{tikzcd} \]
Finally, there is a representation in the middle of the Auslander--Reiten quiver with two separate \( 2 \)-s, which has an explicit realization obtained by combining the two patterns above:
\[ \begin{tikzcd}
F \ar[r, "{\left[ \begin{smallmatrix} 1 \\ 0 \end{smallmatrix} \right]}"] & F^2 \ar[r,"{[1 \; 1]}"] & F \ar[r] & 0 \\
0 \ar[r] \ar[u] & F \ar[u, "{\left[ \begin{smallmatrix} 0 \\ 1 \end{smallmatrix} \right]}"] \ar[r, "{\left[ \begin{smallmatrix} 1 \\ 0 \end{smallmatrix} \right]}"] & F^2 \ar[r,"{[1 \; 1]}"] \ar[u,"{[1 \; 0]}"] & F \ar[u]
\end{tikzcd} \]
In all cases we observe that all matrices have at most one \( 1 \) in each column, and the rest \(0 \)-s. Thus all indecomposables are in the essential image of \( \funct{free}_* \).
\end{proof}

\section{Higher homologies: Additively \( \Z \)-realizable representations}
\label{sec.higherhom}

In this short section we take a look at which linear representations are realizable as a summand of \( \operatorname{H}_n \) of some representation in topological spaces.

The first half of this section is very much parallel to Section~\ref{sec.basic}: Much as we connected additive \( \operatorname{H}_0 \)-realizability to additive \( \Set \)-realizability there, here we get a connection to representations in abelian groups.

We first show that all representations of abelian groups are realizable via \( \operatorname{H}_n \), generalizing \cite{harrington2019stratifying} which in turn generalizes the first version of this type of result from \cite{carlsson2007theory} (actually a version of Proposition~\ref{prop.ZRealizableOverPrime} below).

\begin{proposition}
Let \( X \) be a small category, and \( n \geq 1 \). Then the functor \( (\operatorname{H}_n)_* \colon \rep(X, \Top) \to \rep(X, \mathbf{Ab}) \) is essentially surjective.
\end{proposition}

\begin{proof}
The proof of \cite{harrington2019stratifying}*{Theorems~2.11 and 2.12} applies in this more general setup with only minor adjustments: Let \( M \in \rep(X, \mathbf{Ab}) \). We can always find a projective presentation
\[ \bigoplus_{r \in R} P_{v_r} \to \bigoplus_{g \in G} P_{v_g} \to M \to 0, \]
where \( G \) and \( R \) are sets (which we think about as generators and relations), \( v_g \) and \( v_r \) are objects in \( X \), and \( P_v \) is the projective representation given by \( P_v(w) = \Z X(v,w) \), the free abelian group generated by the set \( X(v,w) \), with structure maps given by post-compositions.

The idea now is to build a representation in \( (n+1) \)-dimensional CW complexes out of the same information. To that end, for any object \( w \in X \) let \( C(w)_0 \) be a single point, and let \( C(w)_n \) be obtained by attaching an \( n \)-cell \( e^n_{g, f} \) to the base point for any generator \( g \in G \) and any morphism \( f \in X(v_g,w) \).

Note that we have constructed a representation \( C_n \) of \( X \) in \( n \)-dimensional CW complexes: for a morphism \( h \in X(w_1, w_2) \) we let \( C_n(h) \) identically map \( e^n_{g, f} \subset C(w_1)_n \) to \( e^n_{g, hf} \subset C(w_2)_n \).

Finally we obtain \( C(w) = C(w)_{n+1} \) by attaching an \( (n+1) \)-cell \( e^{n+1}_{r, h} \) for any relation \( r \in R \) and any morphism \( h \in X(v_r, w) \) in the following way obtained from the projective resolution above.
Consider
\[ P_{v_r}(v_r) \ni \operatorname{id}_{v_r} \mapsto \sum_{g \in G} \sum_{f \in X(v_g, v_r)} \alpha_{g,f,r} f \in \bigoplus_{g \in G} P_{v_g}(v_r). \]
Note that the double sum on the right hand side is finite (as it defines an element of the direct sum). 
In \( C(v_r) \), the boundary of \( e^{n+1}_{r, \operatorname{id}_{v_r}} \) is glued to the \( n \)-skeleton in such a way that for each \( n \)-cell \( e^n_{r, f} \in C(v_r)_n \) the degree of the induced map \( S^n \to S^n \), obtained by composing the attaching map \( S^n \to C(v_r)_n \) with the quotient map \( C(v_r)_n \to S^n \) identifying all other \(n\)-cells of \( C(v_r)_n \) to the base point, is given by \( \alpha_{g,f,r} \).

In \( C(w) \), for a non-identity morphism \( h \in X(v_r, w) \) we glue the boundary of \( e^{n+1}_{r, h} \) to the \( n \)-skeleton by composing the attaching map for \( e^{n+1}_{r, \operatorname{id}_{v_r}} \) with \( C_n(h) \colon C(v_r)_n \to C(w)_n \).

This construction delivers a representation \( C \) of \( X \) in \( (n+1) \)-dimensional CW complexes. It follows from the definition of cellular homology that \( \operatorname{H}_n(C) \) is computed exactly as the cokernel of the projective presentation we started with above, and thus that \( \operatorname{H}_n(C) = M \).
\end{proof}

In particular we note that for any field \( F \) the essential (resp.\ additive) image of \( \operatorname{H}_n(-;F)_* \colon \rep(X, \Top) \to \rep(X, \VectF) \) coincides with the essential (resp.\ additive) image of \( - \otimes_{\Z} F \colon \rep(X, \mathbf{Ab}) \to \rep(X, \VectF) \). We will call representations in this image \emph{(additively) \( \Z \)-realizable}.

Note that the functor \( \funct{free} \colon \Set \to \VectF \) factors through \( \mathbf{Ab} \), so that additive \( \Set \)-realizability is a stronger property than additive \( \Z \)-realizability.

\medskip
Pursuing \( \Z \)-realizability, we get the following analogue of Proposition~\ref{prop.repinfsets}.

\begin{proposition} \label{prop.reduction_to_Ab_fg}
Let \( X \) be a finite category, and \( R \in \rep(X, \vectF) \). If \( R \) is a direct summand of \( S \otimes_{\Z} F \) for some \( S \in \rep(X, \mathbf{Ab}) \), then there is a representation \( S_{\rm f.g.} \) of \( X \) in finitely generated abelian groups such that \( R \) is also a direct summand of \( S_{\rm f.g.} \otimes_{\Z} F \).
\end{proposition}

\begin{proof}
Let \( \iota \colon R \to S \otimes_{\Z} F \) be a split monomorphism. For \( x \) an object of \( X \) we pick a finite basis \( b_1, \ldots, b_r \) of \( R(x) \), and write \( \iota(b_i) = \sum_j s_{i,j} \otimes f_{i,j} \). This gives us a choice of finitely many elements \( s_{i,j} \in S(x) \). Since \( X \) is finite the collection of all these elements for all \( x \) generates a sub-representation \( S_{\rm f.g.} \) of \( S \) which maps all objects to finitely generated abelian groups. By construction \( \iota \) factors through \( S_{\rm f.g.} \otimes_{\Z} F \to S \otimes_{\Z} F \), and thus \( R \) is also a direct summand of \( S_{\rm f.g.} \otimes_{\Z} F \).
\end{proof}

We also note that we have the following result on field extensions, which is proven exactly like the corresponding statement for \( \Set \)-realizable representations in Theorem~\ref{thm.fieldext}.

\begin{theorem}\label{thm.fieldext_for_Z}
Let $R\in \rep(X, \vectF)$ where $X$ is a finite category, and let $F'$ be any field extension of $F$. Then $R$ is additively $\Z$-realizable if and only if $R\otimes_F F' \in \rep(X, \vect_{F'})$ is additively $\Z$-realizable.
\end{theorem}

With the following result, we leave the path parallel to Section~\ref{sec.basic}, and rather observe that additive \( \Z \)-realizability is much more prevalent than additive \( \Set \)-realizability.

\begin{proposition} \label{prop.ZRealizableOverPrime}
Let \( X \) be a finite category, and \( F \) a prime field. Then the functor \( - \otimes_{\Z} F \colon \rep(X, \mathbf{Ab}_{\rm f.g.}) \to \rep(X, \vectF) \) is essentially surjective.
\end{proposition}

\begin{proof}
For the field \( \mathbb{F}_p \) this is obvious.

Let \( R \in \rep(X, \vect_{\mathbb{Q}}) \). Choose a basis of \( R \) (i.e.\ a basis of \( R(x) \) for all objects \( x \) of \( X \)). Consider \( S \) the \( \Z \)-sub-representation generated by this basis. The rank of this representation is the dimension of \( R \), and hence \( R \cong S \otimes_{\Z} \mathbb{Q} \).
\end{proof}

\begin{theorem} \label{thm.char_of_Z_realizable}
Let \( F \) be a field, \( X \) a finite category, and \( R \in \rep(X, \vectF) \). Then \( R \) is additively \( \Z \)-realizable if and only if \( R \) has a presentation where all structure constants are algebraic over the prime field of \( F \).
\end{theorem}

For the proof we need to prepare the following proposition.

\begin{proposition} \label{prop.field_ext_preserve_indec}
Let \( X \) be a finite category, and \( F' \) a field extension of \( F \) such that \( F \) is separably algebraically closed in \( F' \). For \( R \in \rep(X, \vectF) \) indecomposable, the induced representation \( R \otimes_F F' \in \rep(X, \vect_{F'}) \) is also indecomposable.
\end{proposition}

For the proof, we recall the following two results, which, while classical, were not quite present in the authors' minds. 

\begin{theorem}[part of \cite{Cartier_Extensions regulieres}*{Théorème~1}] \label{thm.Cartier}
Let \( F' \) be a field extension of \( F \) such that the algebraic closure of \( F \) in \( F' \) is purely inseparable over \( F \).

Then for any other extension \( E \) of \( F \) all zero-divisors in \( E \otimes_F F' \) are nilpotent.
\end{theorem}

If \( E \) is a finite extension of \( F \), then \( E \otimes_F F' \) additionally is a finite-dimensional commutative \( F \)-algebra, and thus the above theorem tells us that it is a local algebra.

\begin{theorem}[\cite{CurtisReiner}*{Theorem~68.1}] \label{thm.CR}
Let \( D \) be a division algebra over \( F \), with center \( F \), and \( F' \) any field extension of \( F \).

Then \( D \otimes_F F' \) is a division algebra over \( F' \), with center \( F' \).
\end{theorem} 

\begin{proof}[Proof of Proposition~\ref{prop.field_ext_preserve_indec}]
Since \( R \) is indecomposable its endomorphism ring is a local \( F \)-algebra. It suffices to show that also the endomorphism ring of \( R \otimes_F F' \) is local. Note that this ring is isomorphic to \( \End(R) \otimes_F F' \). Moreover, since tensoring preserves the radical, it suffices to show that \( \End(R) / \operatorname{Rad} \End(R) \otimes_F F' \) is a local ring. We write \( D = \End(R) / \operatorname{Rad} \End(R) \), and note that this is a finite-dimensional division algebra over \( F \).

Let us denote the center of \( D \) by \( E \). Then \( E \otimes_F F' \) is a local \( F' \)-algebra by (the remark below) Theorem~\ref{thm.Cartier}. In particular \( E \otimes_F F' / \operatorname{Rad}(E \otimes_F F') \) is a field.

Now by Theorem~\ref{thm.CR} we know that
\[ D \otimes_E \frac{E \otimes_F F'}{\operatorname{Rad}(E \otimes_F F')} \]
is a division algebra. It follows that \( D \otimes_F F' \) is local, and thus ultimately that \( R \otimes_F F' \) is indecomposable.
\end{proof}

We are now ready to prove the main result of this section.

\begin{proof}[Proof of Theorem~\ref{thm.char_of_Z_realizable}]
Assume first that \( R \) is additively \( \Z \)-realizable.

Then, by Proposition~\ref{prop.reduction_to_Ab_fg}, \( R \) is a direct summand of \( S \otimes_{\Z} F \) for some representation \( S \) of \( X \) in finitely generated abelian groups.

Let \( K \) denote the subfield of \( F \) consisting of elements which are algebraic over the prime field. We can decompose \( S \otimes_{\Z} K \) into indecomposables, and this decomposition remains a decomposition into indecomposables over the larger field \( F \) by Proposition~\ref{prop.field_ext_preserve_indec}. In particular \( R \) is defined over \( K \).

\medskip
For the other direction, assume that all matrix entries of the representation \( R \) are algebraic over the prime field \( k \). Clearly these finitely many elements lie in some finite Galois extension field of \( k \), and by Theorem~\ref{thm.fieldext_for_Z} we may assume that \( F \) is this Galois extension. Now one can show that the map
\[ F \otimes_k F \to \bigoplus_{g \in \operatorname{Gal}(F/k)} F^g \colon f_1 \otimes f_2 \mapsto (f_1 g(f_2))_{g \in \operatorname{Gal}(F/k)}, \]
is an isomorphism of \( F \)-\( F \)-bimodules, where \( F^g \) is \( F \), but with multiplication on the right hand side twisted by \( g \): It is immediate that it is a homomorphism of bimodules. Moreover, by \cite{Artin_GaloisTheory}*{Corollary of Theorem~12} the maps \( g \) are linearly independent as endomorphisms of \( F \), and thus the map is injective. Since both sides have the same dimension it is an isomorphism.

In particular \( F \) is a direct summand of \( F \otimes_k F \) as a \( F \)-\( F \)-bimodule. Tensoring with \( R \) we obtain that \( R \) is a direct summand of \( R \otimes_k F \). Now note that since \( R \otimes_k F \) is induced from \( \rep(X, \vect_k) \) we know that it is \( \Z \)-realizable by combining Proposition~\ref{prop.ZRealizableOverPrime} and Theorem~\ref{thm.fieldext_for_Z}.
\end{proof}

\begin{example}
\label{ex:D4-trans}
For \( F = \mathbb{R} \), the representation
\[ \begin{tikzcd}
&&& \mathbb{R}^2 \ar[llld,swap,"{[1\;0]}"] \ar[ld,"{[0\;1]}"] \ar[rd,swap,"{[1\;1]}"] \ar[rrrd,"{[1\;\pi]}"] \\
\mathbb{R} && \mathbb{R} && \mathbb{R} && \mathbb{R}
\end{tikzcd} \]
is not additively \( \Z \)-realizable. Hence it is not additively \( \Set \)-realizable either, and in particular it does not appear as a direct summand in any homology of any diagram of topological spaces.
\end{example}

\section{Discussion}
In this paper we have observed that the complexity of the additive image depends heavily on the orientation of the arrows. For instance, for the $n$-star quiver (or more generally a tree quiver) with all arrows oriented inwards, there are only a finite number of indecomposable representations in the additive image of $\funct{free}_*$. Contrary to this, we have seen that every representation of the $4$-star quiver is in the additive image if the arrows are oriented outwards. This is particularly interesting in light of the fact that this is a quiver of tame representation type. It is therefore natural to wonder if $\funct{free}_*$ is additively surjective for the $5$-star quiver with all arrows oriented outwards. We speculate that this is not the case, but thousands of computational experiments have failed to provide us with an example confirming our suspicion. In Example~\ref{ex.7star}, we gave an example for the case of the $7$-star quiver, so it is plausible that the transition happens when moving from a quiver of tame representation type to a wild quiver. More generally, we have the following question. 
\begin{question}
Does there exist a quiver $Q$ of wild representation type such that $\funct{free}_*$ is additively surjective?
\end{question}
Even for tame quivers one would prefer a way of investigating if all representations are in the additive image without going through an explicit classification.

\medskip
In situations when not all representations are additively realizable one may wonder if the collection of additively realizable ones has any reasonable structure. In fact, just looking at the case of a quiver of type \( D_4 \) with all arrows pointing inwards, one can see that the collection of additively realizable representations is neither closed under kernels, cokernels, or extensions, and does not consist of entire Auslander--Reiten components.
\begin{question}
Does the collection of additively realizable representations have any (reasonable) categorical properties?
\end{question}
One possible approach here could be to think about not only realizable representations, but also realizable morphisms, and ask if there could be better closure properties with respect to these morphisms.

\medskip
Another set of questions pertains to the choice of field: While most of our paper is kept general, we have the strongest results only for finite fields. However it is natural to wonder if similar results also hold in particular for the rational numbers. Let us highlight places where we would be particularly interested in such a generalization:
\begin{question}
Is there an algorithm for deciding if a given rational representation of a category is additively $\Set$-realizable?
\end{question}
Note that the algorithms given in Sections~\ref{sec.algo} and \ref{sec:gset-algorithm} work in this generality in principle, but involve infinite systems of equations and are thus not immediately algorithmically solvable.
\begin{question}
For the quiver of type \( \tilde{D}_4 \) with all arrows oriented outward, which rational representations are additively $\Set$-realizable?
\end{question}
We might suspect that not all indecomposables are additively realizable; however, this suspicion stems solely from our inability to realize them, rather than any clear argument suggesting that they should not be. More generally, it appears challenging to handle scalars that are not roots of unity, although we lack a precise conjecture about how or when such scalars might obstruct additive realizability.

Lastly, our current study of (extended) Dynkin quivers relies partly on computer software and partly on explicit classifications. A similarly arduous approach, making use of explicit classifications, could potentially be employed to precisely identify the extended Dynkin quivers for which $\funct{free}_*$ is additively surjective. However, such an approach would be unsatisfying, and new ideas are needed to develop a deeper conceptual understanding of additive $\Set$-realizability for (extended) Dynkin quivers.

\section*{Acknowledgements}
We are grateful to the reviewer for their careful reading of the manuscript and for their helpful comments. UB and MBB were supported by the German Research Foundation (DFG) through the Collaborative Research Center SFB/TRR 109 Discretization in Geometry and Dynamics– 195170736. MBB was supported by the Dutch Research Council (NWO) through grant~VI.Vidi.233.205. The authors thank the Centre for Advanced Study (CAS) in Oslo for hosting the authors for a period in 2023 as part of the program Representation Theory: Combinatorial Aspects and Applications.

\bibliographystyle{plain} 
\bibliography{refs} 

\newpage\todos

\end{document}